\documentclass[10pt]{article}
	\usepackage{graphicx}
\graphicspath{ {./Lambdas/} }
\usepackage{fancyhdr}
\usepackage[margin=1in, top=1in, bottom=1in]{geometry}
\usepackage{tocloft}
\usepackage{amsmath,amsthm,amsfonts,amssymb}
\usepackage{pxfonts}
\usepackage{physics}
\usepackage{enumitem}
\usepackage{ragged2e}
\usepackage{caption}
\usepackage{subcaption}
\usepackage{float}
\usepackage{mathrsfs}
\usepackage{mathtools}
\usepackage[dvipsnames]{xcolor}
\usepackage[utf8]{inputenc}
\usepackage{cite} 
\fancypagestyle{firstpage}{
\fancyhead{} % removes header only

}

\numberwithin{equation}{section}
\usepackage{docmute} % for inputting complete documents. Useful for LyX users
\usepackage{booktabs} % for better tables
\usepackage{microtype} % for better typography
\usepackage[colorlinks=false, linktoc=none]{hyperref}
\usepackage[dvipsnames]{xcolor}
\usepackage{indentfirst}
\usepackage{mathrsfs}
\usepackage{calrsfs}
\usepackage{amscd}
\usepackage[utf8]{inputenc}
\usepackage{amsthm}
\usepackage{enumitem}
\usepackage{latexsym}
\usepackage{amsfonts}
\usepackage[export]{adjustbox}
\usepackage{xcolor}
\definecolor{gold}{RGB}{255,215,0}

\newtheorem{theorem}{Theorem}[section]

\newtheorem{lemma}{Lemma}[section]

\newtheorem{proposition}{Proposition}[section]

\usepackage{tikz}
\usetikzlibrary{shapes.geometric, arrows}

\tikzstyle{startstop} = [rectangle, rounded corners,
text centered, 
draw=black]

\tikzstyle{diff} = [rectangle, rounded corners,
text centered, 
draw=black]

\title{Global Existence and Diffusive Limits for a Class of Nonlocal Reaction-Diffusion Systems}
\author{Md Shah Alam* and Jeff Morgan**}
\date{}

\begin{document}

\maketitle

\begin{center}
$^{*}$ Department of Mathematics, Huston-Tillotson University, Austin, Texas 78702, malam@htu.edu\\
$^{**}$ Department of Mathematics, University of Houston, Houston, Texas, 77004, jjmorgan@uh.edu
\end{center}
\newcommand{\dt}{\,\mathrm{d}t}
\thispagestyle{firstpage}
%\date{}

\begin{abstract}
We study a class of semilinear reaction-diffusion systems with nonlocal diffusion on a bounded domain $\Omega$ in $\mathbb{R}^n$ with smooth boundary. The initial data is assumed to be component-wise nonnegative and bounded, and the reaction vector field is assumed to be quasi-positive and satisfy a generalized mass control condition. We obtain global existence and uniqueness of component-wise nonnegative solutions, and when the reaction vector field satisfies a linear intermediate sum condition, we establish the uniform boundedness of solutions in $L^p(\Omega)$ for all $2 \le p<\infty$ on bounded time intervals independent of the kernel of the nonlocal diffusion operator. This allows us to generalize a recent diffusive limit result of Laurencot and Walker \cite{laurencot2023nonlocal}. We also analyze a class of $m$-component reaction-diffusion systems in which some of the components diffuse nonlocally and the other components diffuse locally, and establish both global existence and a diffusive limit. 
\end{abstract}

\noindent\rule{\textwidth}{0.4pt}

\vspace{0.1cm}


\medskip


\section{Introduction}\label{Introduction}

We are inspired by recent work of \cite{laurencot2023nonlocal} to investigate the global existence of solutions to $m$-component reaction diffusion systems, and their diffusive limit. Laurencot and Walker considered the nonlocal Gray-Scott model given by 
\begin{equation}\label{eq:intro1}
\left\{
\begin{aligned}
\frac{\partial}{\partial t}u(x,t)&=d_1\int_\Omega \varphi(x,y)(u(y,t)-u(x,t))+f_1(u(x,t),v(x,t)),&&x\in\Omega,t\ge0,\\
\frac{\partial}{\partial t}v(x,t)&=d_2\int_\Omega \varphi(x,y)(v(y,t)-v(x,t))+f_2(u(x,t),v(x,t)),&&x\in\Omega,t\ge0,\\
u&=u_0,\,v=v_0,&&x\in\Omega,t=0,
\end{aligned}
\right.
\end{equation}
where
\begin{equation}\label{eq:intro2}
  f(u,v) =\begin{pmatrix}f_1(u,v)\\f_2(u,v)\end{pmatrix}= \begin{pmatrix}
      -u v^2 + a(1-u) \\
      uv^2 - (a+b) v
  \end{pmatrix},
\end{equation}
$n\ge 2$, $\Omega$ is a bounded domain in $\mathbb{R}^n$ with smooth boundary $\partial\Omega$, $d_1,d_2,a,b>0$, $\varphi:\Omega\times\Omega\to\mathbb{R}_+$ is a measurable function satisfying 
$$\int_\Omega \varphi(x,y)dx=\int_\Omega \varphi(x,y)dy<\infty,$$ and $u_0$ and $v_0$ are bounded nonnegative functions on $\Omega$.  They proved global existence and uniqueness of nonnegative solutions to (\ref{eq:intro1}), and also showed that if $\varphi(x,y)=\psi(|x-y|)$ where $\psi\in C^\infty(\mathbb{R}_+,\mathbb{R}_+)$ has compact support, is non increasing and satisfies
$$\int_{\mathbb{R}^n} \psi(|z|)dz=1\,\text{and }\int_{\mathbb{R}^n} \psi(|z|)|z|^2dz=M<\infty,$$
then solutions to 
\begin{equation}\label{eq:intro3}
\left\{
\begin{aligned}
\frac{\partial}{\partial t}u(x,t)&=d_1\int_\Omega j^{n+2}\psi(j|x-y|)(u(y,t)-u(x,t))+f_1(u(x,t),v(x,t)),&&x\in\Omega,t\ge0,\\
\frac{\partial}{\partial t}v(x,t)&=d_2\int_\Omega j^{n+2}\psi(j|x-y|)(v(y,t)-v(x,t))+f_2(u(x,t),v(x,t)),&&x\in\Omega,t\ge0,\\
u&=u_0,\,v=v_0,&&x\in\Omega,t=0,
\end{aligned}
\right.
\end{equation}
satisfy bounds sufficient to prove a subsequence of the solutions $\{(u^{(j)},v^{(j)})\}$ converge in $L^2(\Omega\times(0,T))$ for every $T>0$ to the unique solution to
\begin{equation} \label{eq:intro4}
\left\{
\begin{aligned}
 \frac{\partial}{\partial t} U  &= \frac{Md_1}{2n}\Delta U + f_1(U,V), &&x\in\Omega,t>0, \\
  \frac{\partial}{\partial t} V  &= \frac{Md_2}{2n}\Delta V + f_2(U,V), &&x\in\Omega,t>0, \\
  \nabla U\cdot \eta&=\nabla V\cdot \eta=0,  &&x\in\partial\Omega,t>0,\\
 U  &= u_0,\,V=v_0  &&x\in\Omega,t=0.
\end{aligned}
\right.
\end{equation}
We should emphasize that the work in \cite{laurencot2023nonlocal} and analysis below is focused on the $n\ge2$ setting. The $n=1$ setting for the model discussed in \cite{laurencot2023nonlocal} is discussed in \cite{cappanera2024analysis} using a Galerkin approximation. 

The vector field given by (\ref{eq:intro2}) has a very special structure. $f$ is quasi-positive. That is, $f_1(0,z),f_2(z,0)\ge0$ whenever $z\ge 0$. It also satisfies a linear intermediate sum condition. That is, there exist $K\ge0$ and $L\in\mathbb{R}$ such that
$$f_1(w,z)\le L(w+z+1),\quad f_1(w,z)+f_2(w,z)\le L(w+z)+K\,\text{for all }w,z\ge 0.$$
Actually, the vector field $f$ above is much better behaved, but we focus below on a more general setting.

Our goal is to prove similar results for $m$-component systems where the reaction vector field $f=(f_i):\mathbb{R}^m\to\mathbb{R}^m$ is locally Lipschitz, satisfies a quasipositivity condition and there exist $c_i>0$, $K\ge0$ and $L\in\mathbb{R}$ so that 
$$\sum_{i=1}^m c_if_i(z)\le L\sum_{i=1}^mz_i+K\,\text{for all }z\in\mathbb{R}_+^m.$$
In this setting, we prove global existence of component-wise nonnegative solutions. Then, with the additional assumptions that (i) the vector field has a polynomial nature (with no restriction on the degree), (ii) there exists an $m\times m$ lower triangular matrix $A=(a_{i,j})$ such that $a_{i,i}=1$ for all $i$, $a_{i,j}>0$ when $i\ge j$, and (iii) there exist constants $K\ge0$ and $L\in\mathbb{R}$ so that for all $k\in\{1,...,m\}$ we have
$$\sum_{j=1}^k a_{k,j}f_j(z)\le L\sum_{j=1}^mz_j+K\,\text{for all }z\in\mathbb{R}_+^m,$$
we obtain a generalization of the diffusive limit result above by deriving uniform $L^p$-bounds for solutions independent of the kernel $\varphi$ of the nonlocal operator. The $L^p$ bounds are obtained by employing an $L^p$ energy functional that was first introduced for 2-component systems with local diffusion in Kouachi \cite{kouachi2001existence} and extended to the $m$-component setting with local diffusion in Abdelmalek and Kouachi \cite{abdelmalek2007proof}. We also obtain global existence and diffusive limit results for systems that combine both nonlocal and local diffusion, with the components satisfying local diffusion also satisfying homogeneous Neumann boundary conditions. In this setting, we combine the $L^p$ energy functional with duality arguments. Precise hypotheses and statements of results are given in section 2.

Nonlocal reaction–diffusion equations (NRDEs) form a broad class of partial differential equations (PDEs) in which the temporal evolution of a quantity (such as concentration, temperature, or population density) is governed by both local and nonlocal interactions. Local dynamics arise through classical diffusion and reaction terms, while nonlocal contributions capture interactions that extend beyond immediate neighborhoods. The inclusion of nonlocal terms allows for more realistic modeling of phenomena where long-range effects are intrinsic to the system. In contrast, classical reaction–diffusion equations (RDEs) account solely for local interactions, assuming that changes at a point are influenced only by its immediate surroundings. While sufficient for many applications, this local assumption can be overly restrictive. NRDEs overcome this limitation by incorporating long-range influences, making them especially well-suited for the study of complex systems in which nonlocal effects play a central role.

Research on the global existence and uniform boundedness of solutions of nonlocal reaction-diffusion systems (NRDs) is very recent compared to research on systems with local diffusion. We refer the reader to interesting recent work in \cite{daoud2024class} and \cite{nguyen2025well} that focus on the setting of nonlocal diffusion associated with the classical fractional Laplacian. The work in \cite{daoud2024class} employs quasipositivity and linear intermediate sums, whereas the work in \cite{nguyen2025well} takes strong advantage of the setting $\mathbb{R}^n$ instead of a bounded domain $\Omega$, and they obtain results in the setting of $f$ being super quadratic and satisfying a mass control condition. We also note the work in  \cite{andreu2010nonlocal, c2022local,tanaka2017mathematical, ahmad2014nonlinear,garcia2009logistic, chasseigne2007dirichlet,torebek2023global, caballero2021existence,gourley1996predator, deng2008nonlocal, deng2015global, liang2022global} for existence and uniqueness and to \cite{kuniya2018global, rodriguez2024nonlinear, chang2022spatiotemporal, pao1995reaction, chasseigne2006asymptotic, ignat2007nonlocal, cortazar2007boundary} for asymptotic behavior and stability. \cite{souplet1998blow, pao1992blowing, marras2017blow, torebek2023global, bogoya2018non, liang2022global} discuss blow-up phenomena of solutions of numerous nonlocal reaction-diffusion systems. 

Our work is organized as follows. Section 2 contains notation, statements of results, and comments on possible extensions. Section 3 contains a few lemmas that are critical to our results. Section 4 contains the proofs of global existence results for $m$-component systems with nonlocal diffusion similar to that given above, and $m$-component systems that mix nonlocal diffusion with classical local diffusion. Section 5 contains the proofs of diffusive limit results, and section 6 provides some examples, along with basic numerical approximations to illustrate some of our results, along with numeric comparisons between approximations of solutions to purely nonlocal systems and those of mixed local–nonlocal systems, highlighting qualitative differences in their pattern formation dynamics. We emphasize that the material in section 6 is only intended for illustrative purposes.

\section{Statements of Main Results and Observations} \label{Main Results}

This section introduces some notation, provides statements of our main results, and gives some observations associated with our work and related problems.
\subsection{Notation}
If $m\in\mathbb{N}$ the nonnegative orthant in $\mathbb{R}^m$ is defined by 
$$\mathbb{R}_+^m=\{z\in \mathbb{R}^m\,|\,z_i\ge0\,\text{for all }i=1,...,m\}.$$
Note that $\mathbb{R}_+=\mathbb{R}_+^1=[0,\infty)$. Throughout, $2\le n\in\mathbb{N}$ and $\Omega$ is a bounded domain in $\mathbb{R}^n$ with smooth boundary $\partial\Omega$ such that $\Omega$ lies on one side of $\partial\Omega$. The smoothness of the boundary does not play a role in our initial results below, but it is necessary in our later results associated with systems with both nonlocal and local diffusion, and also diffusive limits. If $0\le \tau< T$ then $Q_{\tau,T}$ denotes $\Omega\times(\tau,T)$, and $Q_{0,\infty}$ denotes $\Omega\times(0,\infty)$. 
We assume the reader is familiar with the standard spaces of continuous and continuously differentiable functions, as well as $L^p$ and Sobolev spaces. If $1\le p<\infty$ then $W_p^{(1)}(\Omega)$ denotes the Banach space consisting of elements in $L^p(\Omega)$ having the distributional derivatives $\partial_{x_i}u$, where $i\in\{1,...,n\}$, and each of the derivatives lie in $L^p(\Omega)$. The norm on $W_p^{(1)}(\Omega)$ is given by 
$$\|u\|_{p,\Omega}^{(1)}=\|u\|_{p,\Omega}+\sum_{i=1}^n\|\partial_{x_i}u\|_{p,\Omega}.$$
$W_p^{(2)}(\Omega)$ denotes the Banach space consisting of elements in $L^p(\Omega)$ having the distributional derivatives $\partial_{x_i}u$ and $\partial_{x_ix_j}u$, where $i,j\in\{1,...,n\}$, and each of the derivatives lie in $L^p(\Omega)$. The norm on $W_p^{(2)}(\Omega)$ is given by 
$$\|u\|_{p,\Omega}^{(2)}=\|u\|_{p,\Omega}+\sum_{i=1}^n\|\partial_{x_i}u\|_{p,\Omega}+\sum_{i,j=1}^n\|\partial_{x_ix_j}u\|_{p,\Omega}.$$
Similarly, if $0\le \tau< T$ and $1\le p<\infty$ then $W_p^{(2,1)}(Q_{\tau,T})$ denotes the Banach space consisting of elements in $L^p(Q_{\tau,T})$ having the distributional derivatives $\partial_tu$, $\partial_{x_i}u$ and $\partial_{x_ix_j}u$, where $i,j\in\{1,...,n\}$, and each of the derivatives lie in $L^p(Q_{\tau,T})$. The norm on $W_p^{(2,1)}(Q_{\tau,T})$ is given by
$$\|u\|_{p,Q_{\tau,T}}^{(2,1)}=\|u\|_{p,Q_{\tau,T}}+\|\partial_tu\|_{p,Q_{\tau,T}}+\sum_{i=1}^n\|\partial_{x_i}u\|_{p,Q_{\tau,T}}+\sum_{i,j=1}^n\|\partial_{x_ix_j}u\|_{p,Q_{\tau,T}}.$$
Finally, if $X$ is a Banach space, then $X'$ denotes the dual space of $X$. We make use of $\left(W_{n+1}^{(1)}(\Omega)\right)'$ in the proof of our diffusive limit results.

\subsection{Hypotheses and Statements of Results}
Let $m \in \mathbb{N}$ such that $m \ge 2$.  We are initially concerned with the reaction-diffusion system given by
\begin{equation} \label{eq:2.11}
\left\{
\begin{aligned}
 \frac{\partial}{\partial t} u_i(x,t) & = \int_{\Omega} d_i \varphi(x,y)(u_i(y,t)-u_i(x,t)) dy + f_i(u(x,t)), &&x \in \Omega, \ t \ge 0, \ i =1,2,....,m, \\
 u_i(x,0) & = v_i(x), &&x \in \Omega, \ i =1,2,....,m,
\end{aligned}
\right.
\end{equation}
where the following assumptions are satisfied:

\begin{enumerate} [leftmargin=*, labelwidth=4em, labelsep=1em, align=left]
 \item[(INIT)] \hypertarget{item:(INIT)} {(Initial Condition) \ \ $v=(v_i)\in C(\overline{\Omega},\mathbb{R}_+^m)$.}
 \item[(DIFF)] \hypertarget{item:(DIFF)} {(Diffusivity) $d_i>0$ for all $i=1,2,....,m$ and $\varphi\in C\left(\overline\Omega \times \overline\Omega,\mathbb{R}_+\right)$ such that $\varphi(x,y) = \varphi(y,x)$ for all $x,y \in \Omega$ and there exist $\varepsilon,\mu\ge0$ such that
\begin{flalign} \label{eq:2.12}
    0\le\varepsilon\le\int_{\Omega} \varphi(x,y) dy \le \mu\,\text{for all }x\in\Omega.
\end{flalign}}
\item [(F)] \hypertarget{item:(F)} {(Function)\  $f : \mathbb{R}^m \to \mathbb{R}^m$ is locally Lipschitz. Namely, there exists $C\in C(\mathbb{R}^m\times\mathbb{R}^m,\mathbb{R}_+)$ such that 
$$|f(u)-f(v)|\le C(u,v)|u-v|$$ 
for all $u,v\in\mathbb{R}^m$.}
\item[(QP)] \hypertarget{item:(QP)} {(Quasi-positivity) For each $i\in\{1,....,m\}$, $f_i(u) \ge 0 $ whenever $u \in \mathbb{R}_+^m$ such that $u_i=0$.}
 \item[(QBAL)] \hypertarget{item:(QBAL)} {(Quasi-Balancing) There exist $a_i>0$ for $i=1,...,m$ and $K\ge0$ and $L\in\mathbb{R}$ such that 
$$\sum_{i=1}^m a_i\,f_i(u)\le K+L\sum_{i=1}^m u_i\quad\text{for all }u\in\mathbb{R}_+^m.$$}
\end{enumerate}
We remark that the continuity of the functions $v_i$, $f$ and $\varphi$ imply that stating (\ref{eq:2.11}) on $\Omega$ is equivalent to stating (\ref{eq:2.11}) on $\overline\Omega$. If $T>0$ is an extended real number, then a function $u\in C^{(0,1)}(\overline\Omega\times[0,T))$ that satisfies (\ref{eq:2.11}) for $t<T$ is said to be a classical solution. This solution is said to be local if $T<\infty$ and global if $T=\infty$.

The term $\int_{\Omega} d_i \varphi(x,y)(u_i(y,t)-u_i(x,t)) dy$ that appears in (\ref{eq:2.11}) represents a nonlocal diffusion term with kernel $d_i\varphi(x,y)$. The conditions \hyperlink{item:(INIT)}{(INIT)}, \hyperlink{item:(F)}{(F)} and \hyperlink{item:(QP)}{(QP)} above guarantee that (\ref{eq:2.11}) has a unique classical maximal component-wise nonnegative solution. Moreover, as stated below, the condition \hyperlink{item:(QBAL)}{(QBAL)} guarantees the solution is a global solution. We note that \hyperlink{item:(QBAL)}{(QBAL)} has been studied extensively in the setting where nonlocal diffusion in (\ref{eq:2.11}) is replaced by local diffusion in the form of an elliptic operator. When the inequality ``$\le$'' in \hyperlink{item:(QBAL)}{(QBAL)} is replaced by ``$=$'' and $K,L\equiv 0$, this condition is referred to as a balancing condition, or a conservation of mass condition. Many references associated with this can be found in the summary article of Michel Pierre \cite{pierre2010global}. In that setting, many more assumptions are typically required of the reaction vector field $f$ to prove global existence. Things are much simpler in this setting. 

\begin{theorem} \label{global} If \hyperlink{item:(INIT)}{(INIT)}, \hyperlink{item:(DIFF)}{(DIFF)}, \hyperlink{item:(F)}{(F)}, \hyperlink{item:(QP)}{(QP)} and \hyperlink{item:(QBAL)}{(QBAL)} are satisfied, then (\ref{eq:2.11}) has a unique classical component-wise nonnegative global solution $u\in C^{(0,1)}\left(\overline\Omega\times\mathbb{R}_+,\mathbb{R}_+^m\right)$. Furthermore, if $\varepsilon>0$ in \hyperlink{item:(DIFF)}{(DIFF)} and  $K=L=0$, or $L<0$ then the solution $u$ is uniformly sup norm bounded.
\end{theorem}

We remark that the system (\ref{eq:2.11}) can also be posed with nonnegative initial data in $L^\infty(\Omega)$ and the requirement that the equations hold for $x$ a.e. $\Omega$, and the conditions on $\varphi$ changed in \hyperlink{item:(DIFF)}{(DIFF)} so that $\varphi$ is only required to be a symmetric measurable function that satisfies the given integral bound. The proof of the theorem above gives a similar result, with the understanding that the solution no longer lives in $C^{(0,1)}\left(\overline\Omega\times\mathbb{R}_+,\mathbb{R}_+^m\right)$, but instead lives in $C^1([0,\infty),L^\infty(\Omega,\mathbb{R}_+^m))$. However, the uniform sup norm portion of the result seems to require a sup norm bound on $\varphi$.

As stated in the introduction, the primary goal of this work is the generalization of the results in \cite{laurencot2023nonlocal}. The results in that work were indicated in the introduction. In the setting of general $m\in\mathbb{N}$ with $m\ge 2$, we generalize the setting in \cite{laurencot2023nonlocal} as follows.

\begin{enumerate} [leftmargin=*, labelwidth=4em, labelsep=1em, align=left]
\item[(DIFF2)] \hypertarget{item:(DIFF2)} {(Diffusivity) $d_i>0$ for all $i=1,2,....,m$ and  $\psi\in C^\infty\left(\mathbb{R}_+,\mathbb{R}_+\right)$ is nonincreasing, has compact support, $\int_{\mathbb{R}^n}\psi(|z|)dz=1$
and $\int_{\mathbb{R}^n} |z|^2\psi(|z|)dz=M<\infty.$
 \item[(INT-SUM)] \hypertarget{item:(INT-SUM)} {(Intermediate-Sum) There exists a lower triangular $m \times m$ matrix $A=(a_{i,j})$ such that $a_{i,i}=1$ for all $i=1,...,m$ and positive entries below the diagonal, and $L>0$ so that for all $k=1,2,....,m$ we have
\begin{flalign} 
    \sum_{j=1}^k a_{k,j} f_j (u) \le L\left(\sum_{i=1}^k u_i \,+1\right)\ \ \ \text{for all} \ u \in \mathbb{R}_+^m, \label{eq:2.16}
\end{flalign}}}
\item[(Lp Continuity)] \hypertarget{item:(Lp Continuity)} There exists $L\ge 0$ and $r\in\mathbb{N}$ such that
$$ |f_i(u)|\le L\left(\sum_{i=1}^m u_i\,+1\right)^r\quad\text{for all }\,u\in\mathbb{R}_+^m,\,i=1,...,m,$$
and there exists $\tilde{p}\in\mathbb{N}$ such that if $\{v_k\}$ is a Cauchy sequence in $L^p(\Omega,\mathbb{R}^m)$ with $p\ge \tilde{p}$, then 
$\{f_i(v_k)\}$ is a Cauchy sequence in $L^2(\Omega,\mathbb{R}_+^m)$ for all $i=1,...,m$.
\end{enumerate}

We remark that if we had only required that the lower triangular matrix $A$ in \hyperlink{item:(INT-SUM)}{(INT-SUM)} satisfy $a_{i,i}=1$ for all $i$, with no requirement on the entries below the diagonal, then we could readily prove that a modified version of $A$ satisfying the conditions given and a possible larger value of $L$ would give the version given above. In addition, we  remark that \hyperlink{item:(Lp Continuity)}{(Lp Continuity)} holds if each $f_i(u)$ is polynomial in $u$.
	
Similar to the system considered in \cite{laurencot2023nonlocal}, we study the system 
\begin{equation} \label{eq:2.17}
\left\{
\begin{aligned}
 \frac{\partial}{\partial t} u_i(x,t) &= \int_{\Omega} d_i j^{n+2}\psi(j|x-y|)(u_i(y,t)-u_i(x,t)) dy + f_i(u(x,t)), &&x \in \Omega, \ t \ge 0, \ i=1,...,m, \\
 u_i(x,0) &= v_i(x), &&x \in \Omega,\,i=1,...,m.
\end{aligned}
\right.
\end{equation}
We are able to prove that there is a subsequence of the solutions $u^{(j)}$ to (\ref{eq:2.17}) that converges to the solution to
\begin{equation} \label{eq:2.18}
\begin{cases}
 \frac{\partial}{\partial t} U_i  = \frac{Md_i}{2n}\Delta U_i + f_i(U), \ \ \ \text{on }Q_{0,T},\,i=1,...,m,\\
  \nabla U_i\cdot \eta=0, \hspace{0.78 in} \text{on }\partial\Omega\times(0,T),\,i=1,...,m,\\
 U_i  = v_i, \hspace{1.02 in} \text{on } \Omega,t=0,\,i=1,...,m.
\end{cases}
\end{equation}
in $L^p(\Omega\times(0,T))$ for every $T>0$ and $1\le p<\infty$. We remark that the results in \cite{morgan1989global} guarantee the existence of a unique component-wise nonnegative global solution to (\ref{eq:2.18}). 

The precise statement of our diffusive limit result is given below.

 \begin{theorem} \label{DiffLimit} Suppose $n\ge 2$. If \hyperlink{item:(INIT)}{(INIT)}, \hyperlink{item:(DIFF2)}{(DIFF2)}, \hyperlink{item:(F)}{(F)}, \hyperlink{item:(QP)}{(QP)}, \hyperlink{item:(INT-SUM)}{(INT-SUM)} and \hyperlink{item:(Lp Continuity)}{(Lp Continuity)} are satisfied, then (\ref{eq:2.17}) has a unique classical component-wise nonnegative global solution $u^{(j)}=\left(u_i^{(j)}\right)$ for each $j\in\mathbb{N}$, and there exists a unique nonnegative global solution $U=(U_i)$ of (\ref{eq:2.18}). Furthermore, there exists a subsequence of $\{u^{(j)}\}$ that converges in $L^p(Q_{0,T})$ to $U$ for each $T>0$ and $1\le p<\infty$. 
\end{theorem}
 
The proof of Theorem \ref{DiffLimit} involves the use of a complex $L^p$ energy functional that was first introduced for 2-component systems with local diffusion in Kouachi \cite{kouachi2001existence} and extended to the $m$-component setting in Abdelmalek and Kouachi \cite{abdelmalek2007proof}. The estimates obtained from the energy functional are supplemented with the analysis in \cite{laurencot2023nonlocal} to obtain our diffusive limit result.

Similar to the remark made in \cite{laurencot2023nonlocal}, Theorem \ref{DiffLimit} is probably also true when $n=1$, but the compactness argument that works for $n\ge2$ does not work in the setting when $n=1$. The details for $n=1$ are not included here.

The final results in this work are related to systems that couple nonlocal and local diffusion. More precisely, we consider systems of the form 
\begin{equation} \label{eq:2.19}
\left\{
\begin{aligned}
 \frac{\partial}{\partial t} u_i(x,t) &= \Gamma_i (u_i)(x,t) + f_i(u(x,t)), &&x \in \Omega, \ t \ge 0, \ i =1,....,m_1, \\
  \frac{\partial}{\partial t} u_j(x,t) &= d_j \Delta u_j + f_j(u(x,t)), &&x \in \Omega, \ t > 0, \ j =m_1+1,....,m, \\
  \nabla u_j(x,t)\cdot\eta&=0, &&x\in\partial\Omega, \ t> 0, \ j=m_1+1,....,m,\\
 u_i(x,0) &= v_i(x), &&x \in \Omega, i=1,...,m,\\
\end{aligned}
\right.
\end{equation}
where
$$\Gamma_i (u_i)(x,t)=\int_{\Omega} d_i \varphi(x,y)(u_i(y,t)-u_i(x,t)) dy.$$
Here, $m_1,m\in\mathbb{N}$ with $m_1<m$ and $d_i>0$ and $\varphi$ satisfies the conditions in \hyperlink{item:(DIFF)}{(DIFF)}. The functions $v_i$ satisfy a modified version of \hyperlink{item:(INIT)}{(INIT)}, and the functions $f=(f_i)$ satisfy \hyperlink{item:(QP)}{(QP)} a modified version of \hyperlink{item:(INT-SUM)}{(INT-SUM)} for the global existence result, and \hyperlink{item:(INT-SUM)}{(INT-SUM)} for the diffusive limit result. We state these more precisely below.

\begin{enumerate} [leftmargin=*, labelwidth=4em, labelsep=1em, align=left]
 \item[(INIT2)] \hypertarget{item:(INIT2)} {(Initial Condition 2)} $v_i\in C(\overline{\Omega},\mathbb{R}_+)$ for $i\in\{1,...,m_1\}$ and $v_j\in C^2\left(\overline{\Omega},\mathbb{R}_+\right)$ with $\frac{\partial}{\partial\eta}v_j(x)=0$ for all $x\in\partial\Omega$ and $j\in\{m_1+1,...,m\}$.
\item[(INT-SUM2)] \hypertarget{item:(INT-SUM2)} {(Intermediate-Sum 2) 
There exists a lower triangular $m\times m$ matrix $A=(a_{i,j})$ such that $a_{i,i}=1$ for all $i$, $a_{i,j}>0$ for all $i\ge m_1$ and $j\le i$, and nonnegative entries otherwise, and $L\ge 0$ such that for all $k=m_1,...,m$ 
$$\sum_{j=1}^k a_{k,j}f_j(u)\le L\left(\sum_{i=1}^{m} u_i \,+1\right)\quad\text{for all }u\in\mathbb{R}_+^{m}.$$}
\item[(POLY)] \hypertarget{item:(POLY)} There exists $L\ge 0$ and $r\in\mathbb{N}$ such that
$$ |f_i(u)|\le L\left(\sum_{i=1}^m u_i\,+1\right)^r\quad\text{for all }\,u\in\mathbb{R}_+^m,\,i=1,...,m.$$
\end{enumerate}

We remark that if the matrix $A$ in \hyperlink{item:(INT-SUM2)}{(INT-SUM2)} was only required to satisfy $a_{i,i}=1$ and $a_{m_1,j}>0$ for all $j\le m_1$, and no further requirements on the entries below the diagonal in $A$, then a version of $A$ could be created with a potentially larger value of $L$ so that the requirements listed in \hyperlink{item:(INT-SUM2)}{(INT-SUM2)} are met. 

\bigskip By a \emph{solution} of the system (\ref{eq:2.19}), we mean a function \(u\) satisfying the following:
\begin{itemize}
    \item For each \(i=1,\dots,m_1\), we have \( u_i \in C^{(0,1)}\left(\overline{\Omega}\times[0,\infty)\right) \).
    \item For each \(j=m_1+1,\dots,m\), we have \( u_j \in C\left(\overline{\Omega}\times[0,\infty)\right) \cap W_p^{(2,1)}\left(Q_{0,T}\right)\) for each $1<p<\infty$ and $T>0$.
    \item The functions $u_i$ for $i=1,\dots,m$ satisfy the partial differential equations, boundary conditions and initial data in (\ref{eq:2.19}).
\end{itemize}
We have the following result.

\begin{theorem}\label{globalmixed} If \hyperlink{item:(INIT2)}{(INIT2)}, \hyperlink{item:(DIFF)}{(DIFF)}, \hyperlink{item:(F)}{(F)}, \hyperlink{item:(QP)}{(QP)}, \hyperlink{item:(INT-SUM2)}{(INT-SUM2)} and \hyperlink{item:(POLY)}{(POLY)} are satisfied, then (\ref{eq:2.19}) has a unique classical component-wise nonnegative global solution. 
\end{theorem}

We remark that the conditions in \hyperlink{item:(INIT2)}{(INIT2)} can be weakened considerably on the $v_j$ for $j>m_1$ terms by applying the techniques in \cite{hollis1987global}. We leave the details to the interested reader.

\bigskip Finally, we obtain a diffusive limit result associated with (\ref{eq:2.19}). More precisely, for each $j\in\mathbb{N}$ we consider the system (\ref{eq:2.19}) with $\Gamma_i$ replaced by 
\begin{equation}\label{gammaj}
\Gamma_i(u_i)(x,t)=\int_\Omega d_ij^{n+2}\psi(j|x-y|)(u_i(y,t)-u_i(x,t))dy 
\end{equation}
where $\psi\in C^\infty(\mathbb{R}_+,\mathbb{R}_+)$ is nonincreasing, has compact support and  
\begin{equation}\label{psi1}
\int_{\mathbb{R}^n}\psi(|z|)dz=1\,\text{and }\int_{\mathbb{R}^n}\psi(|z|)|z|^2 dz=M.
\end{equation}
 We remark that (\ref{psi1}) can be satisfied even if $\psi$ does not have compact support. For example, with the choice of an appropriate positive constant $C$, the function $\psi(z)=C\exp\left(-|z|^2\right)$ satisfies these conditions.

In this setting, we denote the solutions to (\ref{eq:2.19}) by $u^{(j)}$, and we are able to prove that if $T>0$ then there is a subsequence of the solutions $u^{(j)}$ to (\ref{eq:2.19}) that converges in $L^2(Q_{0,T})$ to the solution to
\begin{equation} \label{eq:2.20}
\left\{
\begin{aligned}
 \frac{\partial}{\partial t} U_i  &= \frac{Md_i}{2n}\Delta U_i + f_i(U), &&Q_{0,T},\,i=1,...,m_1,\\
 \frac{\partial}{\partial t} U_i  &= d_i\Delta U_i + f_i(U), &&Q_{0,T},\,i=m_1+1,...,m,\\
  \nabla U_i\cdot \eta&=0, &&\partial\Omega\times(0,T),\,i=1,...,m,\\
 U_i  &= v_i, &&\Omega,t=0,\,i=1,...,m.
\end{aligned}
\right.
\end{equation}
We remark that similar to the case of (\ref{eq:2.18}) the results in \cite{morgan1989global} guarantee the existence of a unique component-wise nonnegative global solution to (\ref{eq:2.20}).

The precise statement of our diffusive limit result is given below.

 \begin{theorem} \label{DiffLimit2} Suppose $n\ge 2$. If \hyperlink{item:(INIT2)}{(INIT2)}, \hyperlink{item:(DIFF2)}{(DIFF2)}, \hyperlink{item:(F)}{(F)}, \hyperlink{item:(QP)}{(QP)}, \hyperlink{item:(INT-SUM)}{(INT-SUM)}, \hyperlink{item:(Lp Continuity)}{(Lp Continuity)}, (\ref{gammaj}) and (\ref{psi1}) are satisfied, then for each $j\in\mathbb{N}$ (\ref{eq:2.19}) has a unique classical component-wise nonnegative global solution $u^{(j)}=\left(u_i^{(j)}\right)$, and there exists a unique component-wise nonnegative global solution $U=(U_i)$ of (\ref{eq:2.20}). Furthermore, there exists a subsequence of $\{u^{(j)}\}$ that converges in $L^p(Q_{0,T})$ to $U$ for each $T>0$ and $1\le p<\infty$. 
\end{theorem}

The proof of Theorem \ref{DiffLimit2} involves the use of the same complex $L^p$ energy functional that is employed in the proof of Theorem \ref{DiffLimit}, along with a duality argument that makes use of Lemma \ref{lem:2nddiff}. As with the proof of Theorem \ref{DiffLimit}, the estimates obtained from the energy functional and duality are supplemented with the analysis in \cite{laurencot2023nonlocal} to obtain our diffusive limit result.

\subsection{Observations and Extensions}
Careful analysis of the proof of Lemma \ref{Lpbdd} reveals that if $0<s<1$ and $(-\Delta)^s$ is the classical fractional Laplacian, then the lemma can also be obtained for the system
\begin{equation} \label{eq:2.11Laamri}
\left\{
\begin{aligned}
 \frac{\partial}{\partial t} u_i(x,t) & = d_i(-\Delta)^su_i(x,t) + f_i(u(x,t)), &&x \in \Omega, \ t > 0, \ i =1,2,....,m, \\
 u_i(x,t)&=0, && x \in \partial\Omega, \ t > 0, \ i =1,2,....,m, \\
 u_i(x,0) & = v_i(x), &&x \in \Omega, \ i =1,2,....,m.
\end{aligned}
\right.
\end{equation}
with the $L^p(\Omega)$ bounds independent of the choice of $s$. Consequently, the global existence result obtained in \cite{daoud2024class} can be obtained by applying this technique.

We also remark that Theorem \ref{global} can be obtained with a similar proof if the condition \hyperlink{item:(QBAL)}{(QBAL)} is replaced by the existence of functions $h_i\in C^2(\mathbb{R}_+,\mathbb{R}_+)$ such that $h_i(z)\to\infty$ as $z\to\infty$ and $h_i''(z)\ge0$ for all $z\ge 0$, and constants $a_i>0$, $K\ge 0$ and $L\in\mathbb{R}$ so that
$$\sum_{i=1}^m a_i h_i'(u_i)f_i(u)\le K+L\sum_{i=1}^m h_i(u_i)\,\text{for all }u\in\mathbb{R}_+^m.$$
When $K=L=0$ this structure provides a separable convex Lyapunov structure. Note that this implies the functions $w_i=h_i(u_i)$ satisfy
\begin{equation*}
\left\{
\begin{aligned}
\frac{\partial}{\partial t}w_i(x,t)&\le \int_\Omega d_i\varphi(x,y)(w_i(y,t)-w_i(x,t))dy+h_i'(u_i(x,t))f_i(u(x,t)), &&x\in\Omega,t\ge0,i=1,...,m,\\
w_i(x,0)&=h_i(v_i(x)),&&x\in\Omega,
\end{aligned}
\right.
\end{equation*}
so we can follow the proof of Theorem \ref{global} to obtain bounds for $\sum_{i=1}^m a_iw_i(x,t)$ on $Q_{0,T}$ for each $T>0$, and consequently bounds for $u$ from the hypotheses on the functions $h_i$. A well known choice for the functions $h_i$ that occurs in modeling reversible chemical reactions is given by
$$h_i(z)=z\ln\left(\frac{z}{b_i}\right)-z+b_i\,\text{for }z\ge0,$$
where the value at $z=0$ is understood in the limit sense, and $b=(b_i)\in(0,\infty)^m$ is chosen so that $f(b)=\vec{0}$. In the general setting, this structure can also be used to modify \hyperlink{item:(INT-SUM)}{(INT-SUM)} and \hyperlink{item:(INT-SUM2)}{(INT-SUM2)}, and consequently extend the results in Theorems \ref{DiffLimit}, \ref{globalmixed} and \ref{DiffLimit2}.

One further extension can be obtained by noting that the primary consequence of \hyperlink{item:(INT-SUM)}{(INT-SUM)} is the existence of functions $g_j:(0,\infty)^{m-j}\to(0,\infty)$ for $j\in\{1,...,m-1\}$ so that if $\alpha=(\alpha_i)\in(0,\infty)^m$ with $\alpha_m=1$ such that $\alpha_j>g_j(\alpha_{j+1},...,\alpha_m)$ for $j\in\{1,...,m-1\}$, then there exists $L_\alpha>0$ so that 
$$\sum_{i=1}^m \alpha_if_i(u)\le L_\alpha\left(\sum_{i=1}^m u_i +1\right)\,\text{for all }u\in\mathbb{R}_+^m.$$
This is the structure that makes the arguments work in Theorems \ref{DiffLimit} and \ref{DiffLimit2}, and consequently, the results in these theorems can be obtained with \hyperlink{item:(INT-SUM)}{(INT-SUM)} replaced by this condition. We remark that this is equivalent to the condition given in  \cite{abdelmalek2007proof}.  While this structure is not as easily recognized in applications as \hyperlink{item:(INT-SUM)}{(INT-SUM)}, it is possible to construct examples of vector fields that satisfy this condition and do not satisfy \hyperlink{item:(INT-SUM)}{(INT-SUM)}. One example is given by
\begin{equation*}
f(u)=\left(\begin{matrix}u_1u_2^3-u_1^4\\u_1^4-u_1u_2^4\end{matrix}\right)\,\text{for all }u\in\mathbb{R}_+^2,
\end{equation*}
since $f$ clearly does not satisfy \hyperlink{item:(INT-SUM)}{(INT-SUM)}, but if $\alpha\ge1$ then there exists $L_\alpha>0$ so that
$$\alpha\left(u_1u_2^3-u_1^4\right)+\left(u_1^4-u_1u_2^4\right)\le u_1\left(\alpha u_2^3-u_2^4\right)\le L_\alpha u_1\,\text{for all }u\in\mathbb{R}_+^2.$$

We also remark that the proof of Theorem \ref{globalmixed} can be modified using the ideas in the proof of Theorem 2.5 in \cite{morgan1990boundedness} to obtain a uniform boundedness result provided we make a small change in \hyperlink{item:(INT-SUM)}{(INT-SUM)} to require the existence of $c_i>0$ for $i\in\{1,...,m\}$, $K\ge0$ and $\tilde{L}\in\mathbb{R}$ such that
$$\sum_{i=1}^m c_if_i(z)\le \tilde{L}\sum_{i=1}^m z_i+K\,\text{for all }z\in\mathbb{R}_+^m$$
where either $\tilde{L}=K=0$, or $\tilde{L}<0$. We leave the details of this and the above observations to the interested reader.

Furthermore, if $\mu=0$ in \hyperlink{item:(DIFF)}{(DIFF)} then Theorem \ref{globalmixed} gives a global existence result associated with the system  
\begin{equation} \label{eq:2.19zerovarphi}
\left\{
\begin{aligned}
 \frac{\partial}{\partial t} u_i(x,t) &= f_i(u(x,t)), &&x \in \Omega, \ t \ge 0, \ i =1,....,m_1, \\
  \frac{\partial}{\partial t} u_j(x,t) &= d_j \Delta u_j + f_j(u(x,t)), &&x \in \Omega, \ t > 0, \ j =m_1+1,....,m, \\
  \nabla u_j(x,t)\cdot\eta&=0, &&x\in\partial\Omega, \ t> 0, \ j=m_1+1,....,m,\\
 u_i(x,0) &= v_i(x), &&x \in \Omega, i=1,...,m.\\
\end{aligned}
\right.
\end{equation}
In addition,  we can see from careful analysis of the proof of Theorem \ref{globalmixed} that if $a>0$ such that $\|v\|_{\infty,\Omega}\le a$ and $T>0$, then there exists $C>0$ dependent on $T$, $a$, $\mu$ from \hyperlink{item:(DIFF)}{(DIFF)} and the parameters in \hyperlink{item:(INT-SUM)}{(INT-SUM)} and \hyperlink{item:(POLY)}{(POLY)} such that
$$\|u(\cdot,t)\|_{\infty,\Omega}\le C\quad\text{for all }0\le t\le T.$$
Interestingly, it is not possible to obtain a similar result for the system 
\begin{equation} \label{eq:2.19rev}
\left\{
\begin{aligned}
 \frac{\partial}{\partial t} u_i(x,t) &= d_i \Delta u_i(x,t) + f_i(u(x,t)), &&x \in \Omega, \ t > 0, \ i =1,....,m_1, \\
  \frac{\partial}{\partial t} u_i(x,t) &= \Gamma_i (u_i)(x,t) + f_i(u(x,t)), &&x \in \Omega, \ t \ge 0, \ i =m_1+1,....,m, \\
  \nabla u_i(x,t)\cdot\eta&=0, &&x\in\partial\Omega, \ t> 0, \ i=1,....,m_1,\\
 u_i(x,0) &= v_i(x), &&x \in \Omega, i=1,...,m.\\
\end{aligned}
\right.
\end{equation}
The smoothing of the diffusion in the last $m-m_1$ PDEs in (\ref{eq:2.19}) plays an important role, and this is absent from the system in (\ref{eq:2.19rev}). We illustrate this by proving a result for the simple system below.
 \begin{equation}\label{eq:2.19revex}
 \left\{
\begin{aligned}
 \frac{\partial}{\partial t} u(x,t) &= d\Delta u(x,t) - u(x,t)v(x,t)^2, &&x \in \Omega, t > 0,\\
  \frac{\partial}{\partial t} v(x,t) &= \int_\Omega\varphi(x,y)(v(y,t)-v(x,t))dy + u(x,t)v(x,t)^2, &&x \in \Omega, t \ge 0, \\
  \nabla u(x,t)\cdot\eta&=0, &&x\in\partial\Omega, t>0,\\
 u(x,0) &= u_0(x),\,v(x,0)=v_0(x) &&x \in \Omega.\\
\end{aligned}
\right.
\end{equation}
Here $d>0$, $u_0\in C^2(\overline{\Omega},\mathbb{R}_+)$ such that $\nabla u_0(x)\cdot\eta=0$ for all $x\in\partial\Omega$, and $v_0\in C(\overline{\Omega},\mathbb{R}_+)$. It is a simple matter to prove (\ref{eq:2.19revex}) has a unique component-wise nonnegative solution on a maximal interval of existence. The following proposition and proof are motivated by the statement and proof of Proposition 3 in \cite{hollis1987global}.

\begin{proposition}\label{proprev}
Suppose $u_0$ and $v_0$ satisfy the conditions above, $\varphi$ satisfies \hyperlink{item:(DIFF)}{(DIFF)}, $T>0$, $a,b>0$ such that $0<\mu<ab$, and there exists $C_{T,d,\mu,a,b}>0$ so that if $\|u_0\|_{\infty,\Omega}\le a$ and $\|v_0\|_{\infty,\Omega}\le b$ and $u$ and $v$ are the nonnegative solutions to (\ref{eq:2.19revex}) for $0\le t< T$, then $\|u(\cdot,t)\|_{\infty,\Omega},\|v(\cdot,t)\|_{\infty,\Omega}\le C_{T,d,\mu,a,b}$ for all $0\le t< T$. Then $T\le\frac{\ln(ab)-\ln(ab-\mu)}{\mu}$.
\end{proposition}

\noindent Note that if we let $\mu\to0^+$ in Proposition \ref{proprev} then we obtain a result for $\mu=0$ which gives $T\le \frac{1}{ab}$, and this agrees with Proposition 3 in \cite{hollis1987global}.

\begin{proof}
Pick $\tilde{x}\in\Omega$. For each $k\in\mathbb{N}$ let $\tilde{v}_{k}\in C(\overline{\Omega},\mathbb{R}_+)$ such that $\tilde{v}_{k}(x)\le b$ for all $x\in\overline{\Omega}$, $\tilde{v}_{k}(\tilde{x})=b$ and $\tilde{v}_{k}(x)\to0$ as $k\to\infty$ for all $x\ne \tilde{x}$. Now let $u_k$ and $v_k$ be the unique nonnegative solutions to (\ref{eq:2.19revex}) with $u_0\equiv a$ and $v_0=\tilde{v}_{k}$. Note that $u_k\le a$. If we multiply both sides of the PDE for $v_k$ in (\ref{eq:2.19revex}) by $v_k$ and integrate over $\Omega$ we find that
$$\frac{1}{2}\frac{\partial}{\partial t}\int_\Omega v_k(x,t)^2dx\le a C_{T,d,\mu,a,b}\int_\Omega v_k(x,t)^2dx,$$
implying 
$$\int_\Omega v_k(x,t)^2dx\le e^{2aC_{T,\mu,a,b}t}\int_\Omega \tilde{v}_k(x)^2dx\to0\,\text{as }k\to\infty.$$
As a result, since $u_k$ and $v_k$ are sup norm bounded on $Q_{0,T}$ and $\int_\Omega v_k(x,t)^2dx\to0$ as $k\to\infty$ we are assured that $u_kv_k^2\to0$ in $L^p(Q_{0,T})$ for every $1\le p<\infty$. So, standard parabolic estimates imply $\|u_k-a\|_{\infty,\Omega\times(0,T)}\to0$ as $k\to\infty$. Now we use \hyperlink{item:(DIFF)}{(DIFF)} to conclude that
$$\frac{\partial}{\partial t}v_k(x,t)+\mu v_k(x,t)\ge u_k(x,t)v_k(x,t)^2.$$
As a result,
$$\frac{\partial}{\partial t}\left(e^{\mu t}v_k(x,t)\right)\ge e^{-\mu t}u_k(x,t)\left(e^{\mu t}v_k(x,t)\right)^2.$$
So, if we define $V_k(x,t)=e^{\mu t}v_k(x,t)$ and $U_k(x,t)=e^{-\mu t}u_k(x,t)$ then 
$$\frac{\partial}{\partial t}V_k(x,t)\ge U_k(x,t)V_k(x,t)^2,$$
implying
\begin{equation*}
V_k(x,t)\ge \frac{\tilde{v}_k(x)}{1-\tilde{v}_k(x)\int_0^t U_k(x,s)ds},
\end{equation*}
and consequently, since $|V_k|_{\infty,\Omega\times(0,T)}\le e^{\mu T}C_{T,d,\mu,a,b}$, it must be the case that 
\begin{equation}\label{nearTbound}
\tilde{v}_k(x)\int_0^t U_k(x,s)ds<1\,\text{for all }(x,t)\in\overline\Omega\times[0,T). 
\end{equation}
So, if we set $x=\tilde{x}$ and recall $\tilde{v}(\tilde{x})=b$ and $u_k(\tilde{x},t)\to a$ as $k\to\infty$, then
$$U_k(\tilde{x},t)=e^{-\mu t}u_k(\tilde{x},t)\to a e^{-\mu t}\,\text{as }k\to\infty.$$ 
Therefore, (\ref{nearTbound}) implies $\frac{ab}{\mu}\left(1-e^{-\mu T}\right)\le1$, and a simple calculation gives the result.
\end{proof}

Proposition \ref{proprev} does not rule out the possibility of global existence in (\ref{eq:2.19revex}), but it does rule out the possibility of obtaining bounds similar to the ones that can be obtained on finite time intervals in the proof of Theorem \ref{globalmixed}. 

\section{Supporting Lemmas}

In this chapter we state and prove two results, including a basic comparison principle, and then we list several known results that are needed to prove our main results. The set $\Omega$, its boundary $\partial\Omega$ and $\varphi$ satisfy the assumptions of section 2.

\begin{lemma}\label{lem:3.1} Suppose $v,w\in C(\overline\Omega,\mathbb{R})$ and we define 
$$v_-(x)=\left\{\begin{array}{cc}-v(x),&\text{when }v(x)\le0\\0,&\text{when }v(x)>0\end{array}\right..$$
Then
\begin{equation}\label{eq:3.1}
\int_\Omega\int_\Omega v(x)\varphi(x,y)(w(y)-w(x)) \ dydx=-\frac{1}{2}\int_\Omega\int_\Omega (v(y)-v(x))\varphi(x,y)(w(y)-w(x)) \ dydx
\end{equation}
and
\begin{equation}\label{eq:3.2}
\int_\Omega\int_\Omega v_-(x)\varphi(x,y)(v(y)-v(x)) \ dydx \ge 0.
\end{equation}
\end{lemma}

\begin{proof}
We start with (\ref{eq:3.1}),
\begin{flalign*} 
 & \int_{\Omega} \int_{\Omega} v(x)\varphi(x,y)(w(y)-w(x)) \ dydx \\
 & = \frac{1}{2} \int_{\Omega} \int_{\Omega} v(x)\varphi(x,y)(w(y)-w(x)) \ dydx + \frac{1}{2} \int_{\Omega} \int_{\Omega} v(x)\varphi(x,y)(w(y)-w(x)) \ dydx \\
 & = \frac{1}{2} \int_{\Omega} \int_{\Omega} v(x)\varphi(x,y)(w(y)-w(x)) \ dydx + \frac{1}{2} \int_{\Omega} \int_{\Omega} v(y)\varphi(y,x)(w(x)-w(y)) \ dxdy \hspace{0.12 in} [\text{exchanging } x \text{ and } y] \\
 & = \frac{1}{2} \int_{\Omega} \int_{\Omega} v(x)\varphi(x,y)(w(y)-w(x)) \ dydx + \frac{1}{2} \int_{\Omega} \int_{\Omega} v(y)\varphi(x,y)(w(x)-w(y)) \ dydx \hspace{0.2 in}  [\because \ \varphi(x,y)=\varphi(y,x)]\\
 & = -\frac{1}{2} \int_{\Omega} \int_{\Omega} (v(y)-v(x)) \varphi(x,y)(w(y)-w(x)) \ dydx 
\end{flalign*}

To prove (\ref{eq:3.2}), we note that $v(x)=v_+(x) - v_-(x)$ where
$$v_+(x)=\left\{\begin{array}{cc}v(x),&\text{when }v(x)\ge0\\ 0,&\text{when }v(x)<0\end{array}\right..$$
We note that, $v_-(x)=-v(x) \ge 0$ when $v(x) \le 0$. Then, using (\ref{eq:3.1}) we get,
\begin{flalign*} 
\int_{\Omega} \int_{\Omega} v_-(x)\varphi(x,y)(v(y)-v(x)) \ dydx  = -\frac{1}{2} \int_{\Omega} \int_{\Omega} (v_-(y) - v_-(x)) \varphi(x,y)(v(y)-v(x)) \ dydx \\
  =- \frac{1}{2} \int_{\Omega} \int_{\Omega}  (v_-(y) - v_-(x)) \varphi(x,y)(v_+(y) - v_-(y) - v_+(x) + v_-(x)) \ dydx \\
  \ge- \frac{1}{2} \int_{\Omega} \int_{\Omega}  \varphi(x,y)(-v_-(y)^2+v_-(y)v_-(x)+v_-(x)v_-(y)-v_-(x)^2) \ dydx \\
  = \frac{1}{2} \int_{\Omega} \int_{\Omega}  \varphi(x,y)(v_-(y)-v_-(x))^2 \ dydx.
 \end{flalign*}
 
 Since $v_-(y)v_+(y)=0$, $v_-(x)v_+(x)=0$, $v_-(y)v_+(x)\ge0$, $v_-(x)v_+(y)\ge0$ and $\varphi(x,y)\ge0$. Therefore,
 $$\int_{\Omega} \int_{\Omega} v_-(x)\varphi(x,y)(v(y)-v(x)) \ dydx \ge 0.$$
\end{proof}

\begin{lemma} \label{lem:3.2} Suppose $T>0$, $F$ is smooth, $\zeta_0$ is continuous, and $(\zeta,\xi) \in C^1([0,T),C(\overline{\Omega}) \times C(\overline{\Omega}))$ satisfy 
\begin{equation*} 
\left\{
\begin{aligned}
     \zeta_t & \ge K(\zeta)+F(\zeta), &&[0,T),\\
 \zeta & \ge \zeta_0, &&t=0,
\end{aligned}
\right.
\end{equation*}
\textit{and}
\begin{equation*}
\left\{
\begin{aligned}
 \xi_t & = K(\xi)+F(\xi), &&[0,T),\\
 \xi & = \zeta_0, &&t=0,
\end{aligned}
\right.
\end{equation*}
where
\begin{flalign*}
   K(z)(x)=\int_{\Omega} \varphi(x,y)(z(y)-z(x)) \ dy.
\end{flalign*}
Then $\zeta \ge \xi$.
\end{lemma}

\begin{proof} Consider $\sigma=\zeta-\xi$ on $\Omega \times [0,T)$ with $\sigma(0)=\zeta(0) -\xi(0) \ge \zeta_0 - \zeta_0 \ge 0$. 
Then,
\begin{flalign*}
    \sigma_t & = \zeta_t - \xi_t\\
=> \ \sigma_t & \ge K(\zeta) - K(\xi) + F(\zeta) - F(\xi) \\  
=> \ \sigma_t & \ge K(\sigma) + F(\zeta) - F(\xi) \\
=> \ \sigma_t & \ge \int_{\Omega} \varphi(x,y) (\sigma(y,t)-\sigma(x,t))\ dy + F(\zeta) - F(\xi). 
\end{flalign*}
Since, $F$ is smooth, the mean value theorem implies there exists a $\theta(x,t) \in [\zeta(x,t),\xi(x,t)]$ such that
\begin{flalign*}
    F(\zeta(x,t))-F(\xi(x,t))=F'(\theta(x,t))(\zeta(x,t)-\xi(x,t))= L(x,t) \sigma(x,t) 
\end{flalign*}
where $L(x,t)=F'(\theta(x,t))$. 
As a result,
\begin{flalign}\label{eq:3.3}
\sigma_t(x,t) \ge \int_{\Omega} \varphi(x,y) (\sigma(y,t)-\sigma(x,t))\ dy + L(x,t) \sigma(x,t). 
\end{flalign}
Multiply both sides of (\ref{eq:3.3}) by $\sigma_- (x,t)$ and integrate over $\Omega$. Then
\begin{flalign}\label{eq:3.4}
-\frac{1}{2}\frac{d}{dt}\int_\Omega \sigma_-(x,t)^2dx\ge \int_\Omega\int_\Omega \sigma_-(x,t)\varphi(x,y)(\sigma(y,t)-\sigma(x,t))\ dydx-\int_\Omega L(x,t)\sigma_-(x,t)^2dx.
\end{flalign}
So, lemma \ref{lem:3.1} implies
\begin{flalign}\label{eq:3.5}
-\frac{1}{2}\frac{d}{dt}\int_\Omega \sigma_-(x,t)^2dx\ge -\int_\Omega L(x,t)\sigma_-(x,t)^2dx.
\end{flalign}
Now suppose $0<\tau<T$. Then there exists $K>0$ so that $-K\le -L(x,t)$ for all $(x,t)\in\overline\Omega\times[0,\tau]$. As a result,
$$-\frac{1}{2}\frac{d}{dt}\int_\Omega \sigma_-(x,t)^2dx\ge -K\int_\Omega \sigma_-(x,t)^2dx,$$
which implies 
$$\frac{d}{dt}\left(e^{-2Kt}\int_\Omega\sigma_-(x,t)^2dx\right)\le0.$$
Therefore $e^{-2Kt}\int_\Omega\sigma_-(x,t)^2dx\le \int_\Omega\sigma_-(x,0)^2dx= 0$ for all $(x,t)\in\overline\Omega\times[0,\tau]$.
As a result, $\sigma_-\equiv0$ on $\overline\Omega\times[0,\tau]$, which gives $\sigma\ge 0$ on $\overline\Omega\times[0,\tau]$. Since $0<\tau<T$ was arbitrary, the result follows.
\end{proof}

Part (i) of the following result is given in chapter 3, section 5 of \cite{ladyzhenskaia1968linear}, and parts (ii) and (iii) are given in Lemma 3.10 of \cite{morgan1990boundedness}.

\begin{lemma} \label{lem:3.3} \cite{ladyzhenskaia1968linear} Let $D>0$, $0 \le \tau <T$, $\theta \ge 0$ and $\theta \in L^p(Q_{\tau,T})$ with $p>1$ and $\Vert \theta \Vert_{p,Q_{\tau,T}}=1$.Then there exists a unique nonnegative solution $\phi \in W_p^{(2,1)}(Q_{\tau,T})$ to
\begin{equation} \label{eq:3.6}
\left\{
\begin{aligned}
    \frac{\partial \phi}{\partial t} & = - D \Delta \phi - \theta, &&Q_{\tau,T},\\
    \frac{\partial \phi}{\partial \eta} & = 0, &&\partial\Omega \times (\tau,T),\\
    \phi & = 0, \hspace{0.74 in} &&\Omega \times \{T\}.
\end{aligned}
\right.
\end{equation}
In addition, there exists $C_{p,D,T-\tau}>0$ which can be chosen to be nondecreasing in $T-\tau$ (independent of $\theta$) so that
\begin{flalign*} 
\Vert \phi \Vert_{W_p^{(2,1)}(Q_{\tau,T})} \le C_{p,D,T-\tau},
\end{flalign*}
 and in addition:
(i) if $1 < p < \frac{n+2}{2}$ and $1 \le q \le\frac{(n+2)p}{n+2-2p}$, then 
\begin{flalign*} 
 \phi \in L^q(Q_{\tau,T}), \ \text{and} \ \Vert \phi \Vert_{q,Q_{\tau,T}} \le C_{p,D,T-\tau},
\end{flalign*}

(ii) if $1<p<\frac{n+2}{2}$ and $1\le q\le\frac{np}{n+2-2p}$ then 
\begin{flalign*} 
  \Vert \phi(\cdot,t) \Vert_{q,\Omega} \le C_{p,D,T-\tau}\,\text{for all }t\in[\tau,T],
\end{flalign*}

(iii) if $p>\frac{n+2}{2}$ then 
\begin{flalign*} 
  \Vert \phi \Vert_{\infty,Q_{\tau,T}} \le C_{p,D,T-\tau}.
\end{flalign*}
\end{lemma}

The following result is typically an exercise in texts studying Sobolev spaces and finite difference estimates. It is a standard application of the density of smooth functions in Sobolev spaces, along with Taylor's formula with integral remainder and Minkowski's integral inequality.

\begin{lemma}\label{lem:2nddiff}
Suppose $\psi\in C^\infty(\mathbb{R}_+,\mathbb{R}_+)$ such that
$$\int_{\mathbb{R}^n}\psi(|z|)|z|^2dz=M<\infty.$$
If $1<p<\infty$, $p'=\frac{p}{p-1}$, $g\in W_p^{(2)}(\mathbb{R}^n)$ and $w\in L^{p'}(\mathbb{R}^n)$ then 
$$\left|\int_{\mathbb{R}^n} \int_{\mathbb{R}^n}\psi(|x-y|)g(x)(w(y)-w(x))dydx\right|\le \frac{M}{2}\|g\|_{W_p^{(2)}(\mathbb{R}^n)}\|w\|_{p',\mathbb{R}^n}.$$
\end{lemma}

\begin{lemma} \label{lem:3.5} \cite{morgan2023global} Suppose $m \in \mathbb{N},$ $\theta\in\mathbb{R}_+^m$ with $\theta_1, \dots, \theta_{m}>0$, $T>0$, $p\in\mathbb{N}$, and
$u \in C^{(0,1)}(\overline{\Omega} \times [0,T),\mathbb{R}_+^m)$. We define $\mathcal{H}_p[u]$ by
\begin{flalign*}
   \mathcal{H}_p[u](t)=\int_\Omega\sum_{\beta \in \mathbb{Z}_+^{m},\vert \beta \vert = p} \binom{p}{\beta} \theta^{\beta^2} u(x,t)^{\beta}dx\quad\text{for }0\le t<T,
\end{flalign*}
with the convention $u^{\beta}=\Pi_{i=1}^m u_i^{\beta_i}$, $\theta^{\beta^2}=\Pi_{i=1}^m \theta_i^{\beta^2_i}$, $|\beta|=\sum_{i=1}^m \beta_i$, and $\binom{p}{\beta}=\frac{p!}{\beta_1!\beta_2!\dots\beta_m!}$. Then,
\begin{flalign*}
\frac{\partial}{\partial t}\mathcal{H}_0[u](t)=0, \ \frac{\partial}{\partial t}\mathcal{H}_1[u](t)=\int_\Omega\sum_{j=1}^{m} \theta_j \frac{\partial}{\partial t} u_j(x,t)dx,
\end{flalign*}
\textit{and for $p \in \mathbb{N}$ such that $p \ge 2$},
\begin{flalign*}
\frac{\partial}{\partial t} \mathcal{H}_p[u](t)= \int_\Omega\sum_{\vert \beta \vert =p-1} \binom{p}{\beta} \theta^{\beta^2} u(x,t)^{\beta} \sum_{j=1}^{m} \theta_j^{2 \beta_j+1} \frac{\partial}{\partial t} u_j(x,t)dx.
\end{flalign*}
\end{lemma}

\noindent Note that since all $u_i$ are non-negative and $\theta_1, \dots, \theta_{m}>0$, there exist $\alpha_\theta,M_\theta>0$ so that $$\alpha_\theta\|u(\cdot,t)\|_{p,\Omega}\le( \mathcal{H}_p[u](t))^{\frac{1}{p}}\le M_\theta\|u(\cdot,t)\|_{p,\Omega}$$
for all $0\le t<T$.

\begin{lemma} \label{prop:3.1} \cite{andreu2010nonlocal} 
Let $T>0$ and $d>0$. Suppose $\{f_j\}\subset C\left(\overline\Omega\times[0,T]\right)$ such that
\begin{flalign*}
 \lim_{j \to \infty}  \Vert f_j -f \Vert_{L^1(\Omega\times(0,T))}  =0
\end{flalign*}
for some $f \in L^1(Q_{0,T})$. Let $z^0 \in C(\overline\Omega)$, $\psi \in C^{\infty}(\mathbb{R_+,\mathbb{R_+}})$ be a non-increasing function such that $\psi(0)>0$,
\begin{flalign*}
\int_{\mathbb{R}^n} \vert x \vert^2 \psi(|x|) \ dx<\infty,
\end{flalign*}
and define $\varphi_j(x,y)=j^{n+2}\psi(j|x-y|)$ for each $j\in\mathbb{N}$. 
Suppose for $j \in \mathbb{N}$, $z_j \in C^1\left([0,\infty),C(\overline\Omega)\right)$ is the solution to
\begin{equation*}
\left\{
\begin{aligned}
 \frac{\partial}{\partial t} z_j & = d \int_\Omega \varphi_j(x,y)(z_j(y)-z_j(x)) + f_j, &&(x,t)\in  Q_{0,T} \\
 z_j & = z^0, &&(x,t)\in \Omega\times\{0\}.
\end{aligned}
\right.
\end{equation*}
If there exists $z_{\infty} \in L^1(Q_{0,T})$ such that
\begin{flalign*}
 \lim_{j \to \infty} \Vert z_j -z_{\infty} \Vert_{1,Q_{0,T}} =0,
\end{flalign*}
then $z_{\infty}=z$, where $z$ denotes the unique weak solution to
\begin{equation*}
\left\{
\begin{aligned}
  \frac{\partial}{\partial t} z & = D \Delta z + f, &&Q_{0,T},\\
  \frac{\partial}{\partial\eta} z  & = 0, &&\partial\Omega\times(0,T), \\
  z & = z^0, &&\Omega\times\{0\},
\end{aligned}
\right.
\end{equation*}
where 
$$D := \frac{d}{2n} \int_{\mathbb{R}^n} \vert x \vert^2 \psi(x) \ dx.$$
\end{lemma}

\begin{lemma} \label{lem:3.6} \cite{laurencot2023nonlocal}
Assume \hyperlink{item:(DIFF)}{(DIFF)}, define $X=C(\overline\Omega)$, $X_+=\{z\in X | z\ge0\}$, and $\Gamma_{\varphi}(z)$ for $z\in X$ by
$$\Gamma_\varphi(z)=\int_\Omega\varphi(x,y)(z(y)-z(x))dy.$$ 
Then $\Gamma_{\varphi}\in \mathcal{L}(X)$ with $\Vert \Gamma_{\varphi} \Vert_{\mathcal{L}(X)} \le 2 \mu$, and $\Gamma_{\varphi}$ generates a uniformly continuous semigroup $(e^{t\Gamma_{\varphi}})_{t\ge 0}$ on $X$ satisfying 
\begin{flalign} \label{eq:3.8}
    \Vert e^{t\Gamma_{\varphi}}\Vert_{\mathcal{L}(X)} \le 1, \ \ \ \ \ t \ge 0.
\end{flalign}
Moreover, for $t\ge 0$ and $z\in X_+$,
\begin{flalign} \label{eq:3.9}
    e^{t\Gamma_{\varphi}}(z) \ge 0 \quad  \text{and } \quad e^{t\Gamma_{\varphi}} 1=1.
\end{flalign}  
\end{lemma}

\section{Proofs of Global Existence Results}

This section contains the proofs of Theorems \ref{global} and \ref{globalmixed}.

\subsection{Proof of Theorem \ref{global}}

Let $X=C(\overline{\Omega})$ and denote $X^{m}_+ := \{ z\in X^m : z \ge 0 \ \text{on} \ \Omega \}$.  Define $\Gamma_i=\Gamma_{\varphi_i}$ for $i=1,2,....,m$, where $\varphi_i=d_i \varphi$ and $\Gamma_{\varphi_i}$ is defined as in Lemma \ref{lem:3.1}, 
$$A:=\text{diag}\{\Gamma_1,\dots,\Gamma_m\}\quad \text{and}\quad f(z)= (f_i(z)).$$
Now for given $u_0=(u_i(x,0)=v_i(x)) \in X^m_+$ the IVP (\ref{eq:2.11}) is equivalent to
\begin{flalign} 
\frac{d}{dt}u=A(u)+f(u), \ \ \ t \ge 0, \ \ \ u(0)=u_0. \label{eq:4.1}
\end{flalign}
Since $A \in \mathcal{L}(X^m)$ and $f$ is locally Lipschitz continuous (therefore bounded on bounded set of $X^m$), basic semigroup theory implies (\ref{eq:4.1}) has a unique solution 
\begin{flalign*}
     u=(u_i) \in C^1([0,T_{max}),X^m)
\end{flalign*}
defined on a maximal time interval $[0,T_{max})$ with $T_{max} \in (0, \infty]$ and $(t,u_0) \mapsto u(t; u_0)$ defines a semiflow on $X^m$. Furthermore, if $T_{max}<\infty$ then $\limsup_{t\to T_{max}^-}\|u(t)\|_{\infty,\Omega}=\infty$.

We note that $u_i(x,t)$ is the solution of system (\ref{eq:2.11}). Now recall the positive and negative portions of functions as defined in Lemma \ref{lem:3.1}.
Similar to above, the system of initial value problems given by
\begin{equation*}
\left\{
\begin{aligned}
\frac{\partial}{\partial t}U_i(x,t) & = \int_{\Omega} d_i \varphi(x,y) (U_i(y,t)-U_i(x,t)) \ dy + f_i (U_+(x,t)), &&x \in \Omega, \ t \ge 0, \ i=1,2,\dots,m\\
U_i(x,0) & = v_i(x), &&x \in \Omega,\ i=1,2,\dots,m,
\end{aligned}
\right.
\end{equation*}
has a unique solution on an interval $[0,T_{max}^*)$ where $T_{max}^*\in(0,\infty]$.
If we multiply both sides of partial differential equation for $U_{i} (x,t)$ by $U_{i_-}(x,t)$ and integrate over $\Omega$ we find
\begin{flalign*}
\int_\Omega U_{i_-}(x,t) \frac{\partial}{\partial t} U_{i}(x,t) & =  \underbrace{d_i \int_{\Omega} \int_{\Omega} U_{i_-}(x,t) \varphi(x,y) (U_i(y,t)-U_i(x,t)) \ dy \ dx}_{\ge 0 \ \text{by lemma \ref{lem:3.1}}}  + \underbrace{\int_{\Omega} U_{i_-}(x,t) f_i (t,U_+(x,t)) \ dx}_{\ge 0 \ \text{by } \hyperlink{item:(QP)}{(QP)}} \\
=> \ \frac{d}{dt}\int_\Omega U_{i_-}(x,t)^2 \ dx & \le 0, \\
\text{and hence} \ \int_\Omega U_{i_-}(x,t)^2 \ dx & \le \int_\Omega U_{i_-}(x,0)^2 \ dx=0.
\end{flalign*}
Therefore $U_{i_-}(x,t) = 0$ for all $(x,t)\in\overline\Omega\times[0,T_{max}^*)$.
As a result, $U_{i_+}(x,t)=U_i(x,t)$, and consequently, the component-wise nonnegative solution to the above system solves the system (\ref{eq:2.11}). Since, the solution to (\ref{eq:2.11}) is unique, the solution of the system (\ref{eq:2.11}) is component-wise nonnegative.

To get the global solution, we need to show $T_{max}=\infty$. If $T_{max} < \infty$ then
\begin{equation} \label{eq:4.2}
    \limsup_{t \to T^-_{max}} \Vert u(t) \Vert_{\infty,\Omega} = \infty.
\end{equation}
For any $0 \le t < T_{max}$ and $i=1,...,m$ we have,
\begin{flalign*} 
 a_i \frac{\partial}{\partial t}u_{i}(x,t) & = d_i \int_{\Omega} \varphi(x,y) a_i (u_i(y,t)-u_i(x,t)) \ dy + a_i f_i(t,u(x,t)),
\end{flalign*}
so \hyperlink{item:(QBAL)}{(QBAL)} implies
\begin{flalign}\label{ThisOne} 
 \frac{\partial}{\partial t}\sum_{i=1}^m a_i u_{i}(x,t) & \le  \int_{\Omega} \varphi(x,y) \sum_{i=1}^m d_ia_i (u_i(y,t)-u_i(x,t)) \ dy + K+L\sum_{i=1}^m u_i(x,t).
\end{flalign}
Integrating over $\Omega$ and applying the symmetry of $\varphi$ gives
\begin{flalign}\label{L1bound} 
 \frac{\partial}{\partial t}\sum_{i=1}^m a_i \int_\Omega u_{i}(x,t)dx & \le L\alpha \sum_{i=1}^m a_i \int_\Omega u_{i}(x,t)dx+K|\Omega|,
\end{flalign}
where $\alpha= \frac{1}{\max_{i=1,\dots,m}a_i}$ if $L\le0$ and $\alpha= \frac{1}{\min_{i=1,\dots,m}a_i}$ if $L>0$. As a result, 
\begin{flalign}\label{L1bound2} 
 \sum_{i=1}^m a_i \int_\Omega u_{i}(x,t)dx & \le \sum_{i=1}^m  a_i \int_\Omega v_{i}(x)dx+\gamma(K,L,\alpha,\Omega,t),
\end{flalign}
where
\begin{equation*}
\gamma(K,L,\alpha,\Omega,t)=\begin{cases}
K|\Omega|t,&L=0,\\ 
\frac{K|\Omega|\left(e^{L\alpha t}-1\right)}{L\alpha},&L\ne0.
\end{cases}
\end{equation*}
So, if we set $d_{max}=\max_{i\in\{1,...,m\}}{d_i}$ and $d_{min}=\min_{i\in\{1,...,m\}}{d_i}$ then (\ref{L1bound2}), \hyperlink{item:(DIFF)}{(DIFF)} and (\ref{ThisOne}) imply
\begin{flalign}\label{ThisTwo} 
 \frac{\partial}{\partial t}U(x,t)+d_{min} \varepsilon U(x,t)& \le d_{max}\mu\left(\sum_{i=1}^m  a_i \int_\Omega v_{i}(x)dx+\gamma(K,L,\alpha,\Omega,t)\right)  + K+L\alpha U(x,t).
\end{flalign}
where $U(x,t)=\sum_{i=1}^m a_i u_{i}(x,t)$. As a result, a simple comparison principle implies $\|U(\cdot,t)\|_{\infty,\Omega}$ is bounded on $[0,T_{max})$, contradicting the assumption $T_{max}<\infty$. Therefore, $T_{max}=\infty$. 

 Now assume either $\varepsilon>0$ in \hyperlink{item:(DIFF)}{(DIFF)} and $K=L=0$, or $L<0$ in \hyperlink{item:(QBAL)}{(QBAL)}. If $\varepsilon>0$ and $L=K=0$ then (\ref{ThisTwo}) implies
$$\frac{\partial}{\partial t}U(x,t)+d_{min} \varepsilon U(x,t) \le d_{max}\mu\sum_{i=1}^m  a_i \int_\Omega v_{i}(x)dx,$$
implying $\|U\|_{\infty,\Omega\times\mathbb{R}_+}$ is bounded. Similarly, if $L<0$ then (\ref{ThisTwo}) implies
$$\frac{\partial}{\partial t}U(x,t)+|L|\alpha U(x,t) \le d_{max}\mu\left(\sum_{i=1}^m  a_i \int_\Omega v_{i}(x)dx+\frac{K|\Omega|}{|L|\alpha}\right)  + K,$$
again implying $\|U\|_{\infty,\Omega\times\mathbb{R}_+}$ is bounded. In either case, the result follows since $a_i>0$ for all $i$.$\hfill\square$

\subsection{Proof of Theorem \ref{globalmixed}}
 
Let $X=C(\overline{\Omega})$ and note that $\Gamma_i$ is a bounded linear operator on $X$ for each $i=1,...,m_1$. As a result, $S_i(t)=e^{\Gamma_i t}$ is a strongly continuous semigroup on $X$. Also, if we define $A_j(z)=-d_j\Delta z$ for all $j\in\{m_1+1,...,m\}$ and
$$z\in D(A_j) = \big\{w\in W_p^2(\Omega)\,:\,p>n,\,A_j(w)\in C(\overline\Omega),\,\frac{\partial}{\partial\eta}w=0\,\text{on }\partial\Omega \big\}$$
then $A_j:D(A_j)\to X$ is an unbounded linear operator that generates an analytic semigroup on $X$ \cite{stewart1980}, and hence a strongly continuous semigroup $T_j(t)$ on $X$. Consequently, results in \cite{pazy2012semigroups} guarantee that (\ref{eq:2.19}) has a unique maximal solution $u$ on a maximal interval $[0,T_{max})$ where $T_{max}$ is a positive extended real number. If $T_{max}=\infty$ then we are done. Suppose $0<T_{max}<\infty$. Then from \cite{pazy2012semigroups} 
$$\limsup_{t\to T_{max}^-}\sum_{i=1}^{m}\|u_i(\cdot,t)\|_{\infty,\Omega}=\infty.$$
It is straightforward to prove that \hyperlink{item:(QP)}{(QP)} implies the maximal solution $u$ is component-wise nonnegative.

From the assumptions in \hyperlink{item:(INT-SUM2)}{(INT-SUM2)} we know $a_{i,j}>0$ for all $i\in\{m_1,\dots,m\}$ and $j\in\{1,\dots,m\}$ with $i\ge j$. Applying this to the partial differential equations (\ref{eq:2.19}) with $i=m$ and integrating over $\Omega$ gives
$$
\frac{\partial}{\partial t}\int_\Omega \sum_{i=1}^{m} a_{m,i}u_{i} (x,t) dx \le\int_\Omega L  \big(\alpha\sum_{i=1}^{m} a_{m,i}u_{i} (x,t)  + 1 \big)dx,
$$
where $\alpha=\left(\min_{j=1,...,m}a_{m,j}\right)^{-1}$. As a result, there exists $G_{T_{max}}>0$ such that
\begin{equation}\label{L1boundu}
\|u_i(\cdot,t)\|_{1,\Omega}\le G_{T_{max}}\quad\text{for all }t\in [0,T_{max})\quad\text{and }i=1,...,m.
\end{equation}
Since $a_{m_1,i}>0$ for all $i\le m_1$ and we see that
\begin{flalign}
  \frac{\partial}{\partial t} \sum_{i=1}^{m_1} a_{m_1,i}u_{i} (x,t) & =    \sum_{i=1}^{m_1} a_{m_1,i}\Gamma_{i}u_{i} (x,t)+ \sum_{i=1}^{m_1}a_{m_1,i} f_{i} (u(x,t)) \nonumber\\
 & \le d_{max}\int_\Omega \varphi(x,y)\sum_{i=1}^{m_1}a_{m_1,i}u_i(y,t)dy+L\left(\sum_{i=1}^{m} u_{i} (x,t) +1\right) \nonumber\\
 & \le K_{T_{max}}+L\alpha\sum_{i=1}^{m_1} a_{m_1,i}u_{i} (x,t) +L \sum_{j=m_1+1}^{m} u_{j} (x,t) \label{eq:4.4}
 \end{flalign} 
for some $K_{T_{max}},\beta>0$ dependent on $G_{T_{max}}$, $\varphi$ and $d_{max}$. As a result, if $0\le\tau<T<T_{max}$ and $1<p'<\infty$ we can multiply both sides above by 
$\left(\sum_{i=1}^{m_1} a_{m_1,i}u_{i} (x,t)\right)^{p'-1}$
to show
\begin{flalign}
\sum_{i=1}^{m_1} \|u_{i} \| _{p',Q_{\tau,T}} & \le M_{p'} (T-\tau)\left(\sum_{i=1}^m\|u_i(\cdot,\tau)\|_{p',\Omega} + 1 + \sum_{j=m_1+1}^{m} \|u_j\|_{p',Q_{\tau,T}}\right) \label{eq:4.5}
\end{flalign}
for some nondecreasing $M_{p'} \in C(\mathbb{R}_+,\mathbb{R}_+)$ with $M_{p'}(0)=0$.

We will use (\ref{eq:4.5}) to bootstrap the $L^1(\Omega)$ bounds for $u$ to $L^{p'}(\Omega\times(0,T_{max}))$ bounds for all $1<p'<\infty$ by applying duality arguments.  To this end, let $0\le \tau<T< T_{max}$, with $\tau$ further specified below, $j\in\{m_1+1,\dots,m\}$, $1<p<<\frac{n+2}{2}$ (with the proximity to $1$ to be specified below) and $\theta \in L^p(Q_{\tau,T})$ such that $\theta\ge0$ and $\|\theta\|_{p,Q_{\tau,T}}=1$. Let $\phi_j$  be the unique nonnegative solution in $W_p^{(2,1)}(Q_{\tau,T})$ solving
\begin{equation} \label{eq:4.6}
\left\{\begin{aligned}
    \frac{\partial}{\partial t} \phi_{j} & = - d_{j} \Delta \phi_{j} - \theta, &&Q_{\tau,T},\\
    \frac{\partial}{\partial \eta} \phi_{j} & = 0, &&\partial\Omega \times (\tau,T),\\
    \phi_{j} & = 0,&& \Omega \times \{T\}.
\end{aligned}\right.
\end{equation}
At first glance, this might appear to be a backward heat equation, but the substitutions $w_{j}(x,t)=\phi_{j}(x,T+\tau-t)$ and $\tilde\theta(x,t)=\theta(x,T+\tau-t)$ lead to the standard initial boundary value problem given by
\begin{equation} \label{eq:4.7}
\left\{\begin{aligned}
    \frac{\partial}{\partial t} w_{j} & = d_{j} \Delta w_{j} + \tilde\theta, &&Q_{\tau,T},\\
    \frac{\partial}{\partial \eta} w_{j} & = 0,&& \partial\Omega \times (\tau,T),\\
    w_{j} & = 0,&& \Omega \times \{\tau\}.
\end{aligned}\right.
\end{equation}
From Lemma \ref{lem:3.3} there exists $C_{p,d_j,T-\tau}>0$ so that $$\|\phi_j(\cdot,\tau)\|_{p,\Omega},\,\|\phi_j\|_{p,Q_{\tau,T}}^{(2,1)}\le C_{p,d_j,T-\tau},$$ 
along with additional embedding results. We recall that $C_{p,d_j,T-\tau}$ can be chosen to be nondecreasing in $T-\tau$, and since $T_{max}$ is assumed to be finite, we can choose $C_p=\max_{\{j=m_1+1,...,m\}}C_{p,d_j,T_{max}}$, and without loss of generality assume $\|\Delta \phi_j\|_{p,Q_{\tau,T}}\le C_{p}$. Finally, let $p'=\frac{p}{p-1}$ and recall the function $M_{p'}$ from (\ref{eq:4.6}). Note that \hyperlink{item:(INT-SUM2)}{(INT-SUM2)} implies $a_{i,j}>0$ for all $i\ge m_1$ and $j\le i$. Let $k\in\{m_1+1,\dots,m\}$ and set $j=k$ in (\ref{eq:4.6}). In addition, let $\mathbb{I}_k=\{j\in\mathbb{N}\, |\,j\ge m_1+1\,\text{and } j<k\}$. Then
\begin{equation*}
\begin{aligned}
\int_\tau^T\int_\Omega \sum_{i=1}^k a_{k,i}u_i\theta dxdt&=\int_\tau^T\int_\Omega \sum_{i=1}^k a_{k,i}u_i(-d_k\Delta\phi_k-\partial_t\phi_k)dxdt\\
&=-\int_\tau^T\int_\Omega \sum_{i=1}^k a_{k,i}u_i d_k\Delta\phi_kdxdt+\int_\Omega \sum_{i=1}^k a_{k,i}u_i(x,\tau)\phi_k(x,\tau)dx+\int_\tau^T\int_\Omega \phi_k\sum_{i=1}^k a_{k,i}\partial_tu_idxdt\\
&\le-\int_\tau^T\int_\Omega \sum_{i=1}^{m_1} a_{k,i}u_i d_k\Delta\phi_kdxdt+\int_\Omega \sum_{i=1}^k a_{k,i}u_i(x,\tau)\phi_k(x,\tau)dx\\
&+\int_\tau^T\int_\Omega \phi_k(x,t)\sum_{i=1}^{m_1} a_{k,i}\int_\Omega d_i\varphi(x,y)(u_i(y,t)-u_i(x,t))dydxdt\\
&+\int_\tau^T\int_\Omega \sum_{i\in \mathbb{I}_k} a_{k,i} (d_i-d_k)u_i\Delta\phi_k dxdt+\int_\tau^T\int_\Omega \phi_k L\left(\sum_{i=1}^{m} u_i+1\right)dxdt\\
&=I_1+I_2+I_3+I_4+I_5.
\end{aligned}
\end{equation*}
 If $A_{max}=\max_{i,j}a_{i,j}$ then (\ref{eq:4.5}) implies
\begin{equation}\label{I1-bound}
\begin{aligned}
|I_1|&\le A_{max}d_{max}\int_\tau^T\int_\Omega \sum_{i=1}^{m_1} u_i |\Delta\phi_k|dxdt\le A_{max}d_{max}C_p\sum_{i=1}^{m_1}\|u_i\|_{p',Q_{\tau,T}}\\
&\le A_{max}d_{max}C_pM_{p'} (T-\tau)\left(\sum_{i=1}^m\|u_i(\cdot,\tau)\|_{p',\Omega} + 1 + \sum_{j=m_1+1}^{m} \|u_j\|_{p',Q_{\tau,T}}\right).
\end{aligned}
\end{equation}
From Lemma \ref{lem:3.3}
\begin{equation}\label{I2-bound}
\begin{aligned}
I_2=\int_\Omega \sum_{i=1}^k a_{k,i}u_i(x,\tau)\phi_k(x,\tau)dx\le A_{max}C_p\sum_{i=1}^m\|u_i(\cdot,\tau)\|_{p',\Omega}.
\end{aligned}
\end{equation}
From Lemma \ref{lem:3.3} and (\ref{L1boundu})
\begin{equation}\label{I3-bound}
\begin{aligned}
I_3&\le A_{max}d_{max}\int_\tau^T\int_\Omega \phi_k(x,t)\left(\int_\Omega\sum_{i=1}^{m_1} \varphi(x,y)u_i(y,t)dy\right)dxdt\\
&\le A_{max}d_{max}\|\varphi\|_{\infty,\Omega\times\Omega}m_1G_{T_{max}}C_p\left(|\Omega|T_{max}\right)^{1/p}.
\end{aligned}
\end{equation}
\begin{equation}\label{I4-bound}
\begin{aligned}
I_4&\le A_{max}d_{max}\int_\tau^T\int_\Omega \sum_{i\in\mathbb{I}_k}u_i|\Delta \phi_k|dxdt\\
&\le A_{max}d_{max}C_p\sum_{i\in\mathbb{I}_k}\|u_i\|_{p',Q_{\tau,T}}.
\end{aligned}
\end{equation}
Now recall $1<p<\frac{n+2}{2}$. If $q=\frac{(n+2)p}{n+2-2p}$ then Lemma \ref{lem:3.3} implies
$$\|\phi_k\|_{1,Q_{\tau,T}}\le C_p.$$ 
So,
\begin{equation}\label{I5-bound1}
\begin{aligned}
I_5&=\int_\tau^T\int_\Omega \phi_k L\left(\sum_{i=1}^{m} u_i+1\right)dxdt\\
&\le LC_p\left(|\Omega|(T_{max}-\tau)|\right)^{1/p}+L\int_\tau^T\int_\Omega \phi_k \sum_{i=1}^{m} u_idxdt.
\end{aligned}
\end{equation}
Note that if $i\in\{1,...,m\}$ then
\begin{equation*}
\begin{aligned}
\int_\tau^T\int_\Omega \phi_k u_idxdt&\le \|\phi_k\|_{q,Q_{\tau,T}}\left(\int_\tau^T\int_\Omega u_i^{\frac{p(n+2)}{(p-1)(n+2)+2p}}dxdt\right)^{\frac{p-1}{p}+\frac{2}{n+2}}\\
&\le C_p\left(\int_\tau^T\int_\Omega u_i^{\frac{p(n+2)}{(p-1)(n+2)+2p}}dxdt\right)^{\frac{p-1}{p}+\frac{2}{n+2}}
\end{aligned}
\end{equation*}
It is straightforward to show that $\frac{p(n+2)}{(p-1)(n+2)+2p}<\frac{n+2}{2}$. Also, if $p$ is sufficiently close to $1$ and we define
$$k=\frac{2p^2}{(p-1)(n+2)+2p}$$
then it can be shown that $\frac{1}{k}>1$,
$$\left(\frac{p(n+2)}{(p-1)(n+2)+2p}-k\right)\cdot\frac{1}{1-k}=\frac{p}{p-1}=p'$$
and
$$\left(\frac{p-1}{p}+\frac{2}{n+2}\right)(1-k)<\frac{p-1}{p}=\frac{1}{p'}.$$
Therefore, if we define $\delta_p=p'\left(\frac{p-1}{p}+\frac{2}{n+2}\right)(1-k)$ then $0<\delta_p<1$ and
\begin{equation*}
\begin{aligned}
\left(\int_\tau^T\int_\Omega u_i^{\frac{p(n+2)}{(p-1)(n+2)+2p}}dxdt\right)^{\frac{p-1}{p}+\frac{2}{n+2}}&=\left(\int_\tau^T\int_\Omega u_i^{\frac{p(n+2)}{(p-1)(n+2)+2p}-k}u_i^kdxdt\right)^{\frac{p-1}{p}+\frac{2}{n+2}}\\
&\le G_{T_{max}}^{\frac{2p}{n+2}}\|u_i\|_{p',Q_{\tau,T}}^{\delta_p}
\end{aligned}
\end{equation*}
from the inequalities above and the fact that the conjugate of $\frac{1}{k}$ is $\frac{1}{1-k}$. Consequently, we see from (\ref{I5-bound1}) that
\begin{equation}\label{I5-bound}
\begin{aligned}
I_5&\le LC_p\left(|\Omega|(T_{max}-\tau)|\right)^{1/p}+LG_{T_{max}}^{\frac{2p}{n+2}}\sum_{i=1}^m \|u_i\|_{p',Q_{\tau,T}}^{\delta_p}.
\end{aligned}
\end{equation}
Combining these estimates gives
\begin{equation}\label{ineq-k}
\begin{aligned}
\int_\tau^T\int_\Omega \sum_{i=1}^k u_i\theta dxdt&\le I_1+I_2+I_3+I_4+I_5\\
&\le K_{p,T_{max}}\left(1+\sum_{i=1}^m\|u_i(\cdot,\tau)\|_{p',\Omega}+\sum_{i\in\mathbb{I}_k}\|u_i\|_{p',Q_{\tau,T}}\right.\\
&\left.+M_{p'}(T_{max}-\tau)\sum_{j=m_1+1}^m\|u_j\|_{p',Q_{\tau,T}}+\sum_{i=1}^m\|u_i\|_{p',Q_{\tau,T}}^{\delta_p}\right)
\end{aligned}
\end{equation}
for some constant $K_{p,T_{max}}>0$ which also depends on $d_{max}$, $|\Omega|$, $A_{max}$, $A_{min}$ and $\varphi$ where
$$A_{min}=\min_{r,i\in\{1,...,m\},\,r\ge m_1,\,r\ge i}a_{r,i}>0.$$
Note that when $k=j=m_1+1$ the set $\mathbb{I}_k=\emptyset$, and duality implies
\begin{equation}\label{ineq-k1}
\begin{aligned}
\|u_i\|_{p',Q_{\tau,T}}&\le K_{p,T_{max}}\left(1+\sum_{i=1}^m\|u_i(\cdot,\tau)\|_{p',\Omega}+M_{p'}(T_{max}-\tau)\sum_{j=m_1+1}^m\|u_j\|_{p',Q_{\tau,T}}+\sum_{i=1}^m\|u_i\|_{p',Q_{\tau,T}}^{\delta_p}\right)
\end{aligned}
\end{equation}
for all $i\in\{1,...,m_1+1\}$. It is clear that if $m_1+1<m$ and we set $k=j=m_1+2$ then $\mathbb{I}_k=\{m_1+1\}$, and (\ref{ineq-k}) and (\ref{ineq-k1}) can be combined to prove there exists a value $K_{p,T_{max}}>0$ which is potentially larger than the one in (\ref{ineq-k}) and (\ref{ineq-k1}) so that 
\begin{equation}\label{ineq-k2}
\begin{aligned}
\|u_i\|_{p',Q_{\tau,T}}&\le K_{p,T_{max}}\left(1+\sum_{i=1}^m\|u_i(\cdot,\tau)\|_{p',\Omega}+M_{p'}(T_{max}-\tau)\sum_{j=m_1+1}^m\|u_j\|_{p',Q_{\tau,T}}+\sum_{i=1}^m\|u_i\|_{p',Q_{\tau,T}}^{\delta_p}\right)
\end{aligned}
\end{equation}
for all $i\in\{1,...,m_1+2\}$. Proceeding inductively, we can prove there exists $K_{p,T_{max}}>0$ sufficiently large so that 
\begin{equation*}
\begin{aligned}
\|u_i\|_{p',Q_{\tau,T}}&\le K_{p,T_{max}}\left(1+\sum_{i=1}^m\|u_i(\cdot,\tau)\|_{p',\Omega}+M_{p'}(T_{max}-\tau)\sum_{j=m_1+1}^m\|u_j\|_{p',Q_{\tau,T}}+\sum_{i=1}^m\|u_i\|_{p',Q_{\tau,T}}^{\delta_p}\right)
\end{aligned}
\end{equation*}
for all $i\in\{1,...,m\}$. Consequently, 
\begin{equation*}
\begin{aligned}
\sum_{i=1}^m\|u_i\|_{p',Q_{\tau,T}}&\le m K_{p,T_{max}}\left(1+\sum_{i=1}^m\|u_i(\cdot,\tau)\|_{p',\Omega}+M_{p'}(T_{max}-\tau)\sum_{j=m_1+1}^m\|u_j\|_{p',Q_{\tau,T}}+\sum_{i=1}^m\|u_i\|_{p',Q_{\tau,T}}^{\delta_p}\right).
\end{aligned}
\end{equation*}
So, if $\tau<T<T_{max}$ is chosen so that $m K_{p,T_{max}}M_{p'}(T_{max}-\tau)\le \frac{1}{2}$ then
\begin{equation}\label{ineq-km}
\begin{aligned}
\sum_{i=1}^m\|u_i\|_{p',Q_{\tau,T}}&\le 2mK_{p,T_{max}}\left(1+\sum_{i=1}^m\|u_i(\cdot,\tau)\|_{p',\Omega}+\sum_{i=1}^m\|u_i\|_{p',Q_{\tau,T}}^{\delta_p}\right).
\end{aligned}
\end{equation}
Then, since $0<\delta_p<1$ we can conclude that there exists a potentially larger value $K_{p,T_{max}}>0$ so that
$$\sum_{i=1}^m\|u_i\|_{p',Q_{\tau,T}}\le 2mK_{p,T_{max}}\left(1+\sum_{i=1}^m\|u_i(\cdot,\tau)\|_{p',\Omega}\right),$$
and as a result, 
$$\sum_{i=1}^m\|u_i\|_{p',Q_{\tau,T_{max}}}\le 2mK_{p,T_{max}}\left(1+\sum_{i=1}^m\|u_i(\cdot,\tau)\|_{p',\Omega}\right).$$
Therefore, since $\tau<T_{max}$, we can conclude that for every $1\le p'<\infty$ that $\|u_i\|_{p',Q_{0,T_{max}}}$ is bounded for each $i\in\{1,...,m\}$.

\medskip Now we apply \hyperlink{item:(POLY)}{(POLY)} to each PDE for $u_j$ for $j=m_1+1,...,m$. Recall that there exists $L>0$ and $r\in\mathbb{N}$ such that
$$f_j(z)\le L\left(\sum_{i=1}^{m}z_i+1\right)^r\quad\text{for all }z\in\mathbb{R}_+^{m}.$$
Let 
$$G(z)=L\left(\sum_{i=1}^{m}z_i+1\right)^r\quad\text{for all }z\in\mathbb{R}_+^{m}.$$
Then the estimates above imply $\|G(u)\|_{p,Q_{\tau,T_{max}}}$ is bounded for each $1<p<\infty$. Now, for each $j=m_1+1,...,m$ let $w_j$ be the unique nonnegative solution to
\begin{equation} \label{eq:4.8c}
\left\{\begin{aligned}
    \frac{\partial}{\partial t} w_{j} & = d_{j} \Delta w_{j} + G(u), &Q_{\tau,T_{max}},\\
    \frac{\partial}{\partial \eta} w_{j} & = 0,& \partial\Omega \times (0,T_{max}),\\
    w_{j} & = v_j,& \Omega \times \{0\}.
\end{aligned}\right.
\end{equation}
The comparison principle for parabolic initial boundary value problems implies $u_j\le w_j$ for $j\in\{m_1+1,...,m\}$, and from the $L^p(Q_{\tau,T_{max}})$ bound on $G(u)$ for $p>\frac{n+2}{2}$,  $\|w_j\|_{\infty,Q_{\tau,T_{max}}}$ is bounded. Now we can return to the PDEs for the $u_i$ with $i\in\{1,...,m_1\}$ and argue as in the proof Theorem \ref{global} to conclude that $\|u_i\|_{\infty,Q_{\tau,T_{max}}}$ is bounded for all $i=1,...,m_1$. Consequently, 
$$\limsup_{t\to T_m^-}\sum_{i=1}^{m}\|u_i(t,\cdot)\|_{\infty,\Omega}<\infty,$$
contradicting the assumption that $T_{max}<\infty$. Therefore, $T_{max}=\infty$. \hfill$\square$

\section{Diffusive Limit - Proofs of Theorem \ref{DiffLimit} and \ref{DiffLimit2}}

We observe that the sup norm bounds we obtained in our proofs of Theorems \ref{global} and \ref{globalmixed} on $Q_{\tau,T}$ depend on the kernel of the convolution operator. Since our ultimate goal is to obtain a diffusive limit, it is important to obtain bounds that are independent of the kernel. In this chapter we use \hyperlink{item:(QP)}{(QP)} and \hyperlink{item:(INT-SUM)}{(INT-SUM)} to develop bounds for solutions of the system (\ref{eq:2.11}) in $L^p(\Omega)$ for $p\in\mathbb{N}$ with $p \ge 2$ that are independent of the kernel function $\varphi$. 

\subsection{$L^p$ Bounds Independent of $\varphi$}
 To construct our $L^p$-energy functional, we write $\mathbb{Z}_{+}^{m}$ for the set of all $m$-tuples of non negative integers. Addition and
scalar multiplication by non negative integers of elements in $\mathbb{Z}_{+}^{m}$
is understood in the usual manner. If $\beta=(\beta_{1},...,\beta_{m})\in \mathbb{Z}_{+}^{m}$ and $p\in \mathbb N$,
then we define $\beta^{p}=((\beta_{1})^{p},...,(\beta_{m})^{p})$.
Also, if $\alpha=(\alpha_{1},...,\alpha_{m})\in  \mathbb{Z}_{+}^{m}$, then
we define $|\alpha|=\sum_{i=1}^{m}\alpha_{i}$. Finally, if  $u=(u_1,...,u_m)\in\mathbb{R}_+^m$ then $u^{\beta}=\Pi_{i=1}^m u_i^{\beta_i}$, where we interpret $0^{0}$ to be $1$, $\theta^{\beta^2}=\Pi_{i=1}^m \theta_i^{\beta^2_i}$, and $\binom{p}{\beta}=\frac{p!}{\beta_1!\beta_2!\dots\beta_m!}$. Let $u=(u_1,...,u_m)$ be the solution to (\ref{eq:2.11}). For $2\leq p\in \mathbb N$, we build our $L^p$-energy function given by 
\begin{equation}\label{Lp}
	\mathcal{L}_p[u](t) = \int_\Omega \mathcal{H}_p[u](x,t)dx
\end{equation}
where
\begin{equation}\label{Hp}
	\mathcal{H}_p[u](x,t) = \sum_{\beta\in \mathbb Z_+^{m}, |\beta| = p}\begin{pmatrix}
	p\\ \beta\end{pmatrix}\theta^{\beta^2}u(x,t)^{\beta}
\end{equation}
was introduced in Lemma \ref{lem:3.5}, and the positive constants $\theta= (\theta_1,\ldots, \theta_{m})$ are to be chosen. We again remind the reader that this energy functional first appeared in \cite{kouachi2001existence}. For convenience, hereafter we drop the subscript $\beta\in \mathbb Z_+^{m}$ in the sum as it should be clear.

Recall from Lemma \ref{lem:3.5}, if $2\leq p\in \mathbb N$ and $\theta \in (0,\infty)^m$ then there exist $0<\varepsilon_1<\varepsilon_2$ dependent on the choice of $\theta$ so that 
\begin{equation}\label{equivnorm}
\varepsilon_1\|u(\cdot,t)\|_{p,\Omega}\le\left(\mathcal{L}_p[u](t)\right)^{\frac{1}{p}}\le \varepsilon_2\|u(\cdot,t)\|_{p,\Omega}.
\end{equation}
We use this fact to get the $L^p$ bounds below.

\begin{lemma} \label{Lpbdd}
Suppose \hyperlink{item:(INIT)}{(INIT)}, \hyperlink{item:(DIFF)}{(DIFF)}, \hyperlink{item:(F)}{(F)}, \hyperlink{item:(QP)}{(QP)} and \hyperlink{item:(INT-SUM)}{(INT-SUM)}, and $u$ is the unique component-wise global solution to (\ref{eq:2.11}). If $T>0$ and $p \in \mathbb{N}$ such that $p \ge 2$ then there exists $L_{p}\in C(\mathbb{R}_+,\mathbb{R}_+)$ independent of the choice of $\varphi$ in \hyperlink{item:(DIFF)}{(DIFF)} such that
$$\|u(\cdot,t)\|_{p,\Omega}\le L_{p}(t)\,\text{for all }t\ge0.$$
\end{lemma}

\textbf{Proof of Lemma \ref{Lpbdd}:} Let $2\leq p\in \mathbb N$. From results in \cite{morgan2023global}, it is possible to choose $\theta\in(0,\infty)^m$ so that 
\begin{enumerate}[label=($\theta$\theenumi),ref=$\theta$\theenumi]
	\item\label{theta1} The matrix $m\times m$ matrix $\mathscr{M}=\left(\mathscr{M}_{i,j}\right)$ is positive
	definite, where 
	\begin{equation}\label{M}
	\mathscr{M}_{i,j}=\begin{cases}
	\begin{array}{cc}
	d_{i}\theta_{i}^{2}, & \text{ if \ensuremath{i=j}}\\
	\frac{d_{i}+d_{j}}{2}, & \text{ if \ensuremath{i\ne j}}
	\end{array}\end{cases}
	\end{equation}
	
	\item\label{theta2} There exists $C_p>0$ such that
	$$\int_{\Omega} \sum_{\vert \beta \vert = p-1} \binom{p}{\beta} \theta^{\beta^2} u^{\beta} \sum_{i=1}^m \theta_i^{2\beta_i+1}  f_i(u) \ dx\le C_p\left(1+\mathcal{L}_p[u](t)\right).$$
\end{enumerate}
In addition, results in \cite{morgan2023global} guarantee
\begin{equation*}
(\mathcal{L}_p[u])'(t)  = \int_{\Omega} \sum_{\vert \beta \vert = p-1} \binom{p}{\beta} \theta^{\beta^2} u^{\beta} \sum_{i=1}^m \theta_i^{2\beta_i+1} \frac{\partial}{\partial t} u_i \ dx.
\end{equation*}
For convenience, we write $\varphi_i=d_i\varphi$. As a result,
\begin{flalign*}
(\mathcal{L}_p[u])'(t)& = \int_{\Omega} \sum_{\vert \beta \vert = p-1} \binom{p}{\beta} \theta^{\beta^2} u^{\beta} \sum_{i=1}^m \theta_i^{2\beta_i+1} \big( \int_{\Omega}\varphi_i(x,y) (u_i(y,t)-u_i(x,t)) \ dy +f_i(u) \big) \ dx\\
& = \int_{\Omega} \sum_{\vert \beta \vert = p-1} \binom{p}{\beta} \theta^{\beta^2} u^{\beta} \sum_{i=1}^m \theta_i^{2\beta_i+1} \big( \int_{\Omega}\varphi_i(x,y) (u_i(y,t)-u_i(x,t)) \big) \ dy \ dx\\
& + \int_{\Omega} \sum_{\vert \beta \vert = p-1} \binom{p}{\beta} \theta^{\beta^2} u^{\beta} \sum_{i=1}^m \theta_i^{2\beta_i+1}  f_i(t,u) \ dx \le I+C_p\left(1+\mathcal{L}_p[u](t)\right) \end{flalign*}
where
\begin{flalign*}
I & = \int_{\Omega} \sum_{\vert \beta \vert = p-1} \binom{p}{\beta} \theta^{\beta^2} u^{\beta} \sum_{i=1}^m \theta_i^{2\beta_i+1} \big( \int_{\Omega}\varphi_i(x,y) (u_i(y,t)-u_i(x,t))  \ dy\big) \ dx.
\end{flalign*}
We note that, for any $\alpha>0,$ using a similar estimate as in (\ref{eq:3.1}), we get,
\begin{flalign*}
& \int_{\Omega \times \Omega} u(x,t)^{\alpha} \varphi_i (x,y) (u(y,t)-u(x,t)) \ d(x,y)= - \frac{1}{2} \int_{\Omega \times \Omega} (u(y,t)^{\alpha} -u(x,t)^{\alpha}) \varphi_i (x,y) (u(y,t) -u(x,t)) \ d(x,y)
\end{flalign*}
where we have used the symmetry of $\varphi(x,y)=\varphi(y,x)$. Applying this technique to $I$ we get,
\begin{flalign*}
I & = -\frac{1}{2} \int_{\Omega \times \Omega} \sum_{\vert \beta \vert = p-1} \binom{p}{\beta} \theta^{\beta^2} (u(y,t)^{\beta} -u(x,t)^{\beta}) \sum_{i=1}^m \theta_i^{2\beta_i+1} \varphi_i (x,y) (u_i(y,t)-u_i(x,t)) \ d(x,y)
\end{flalign*}
Now, for a given $\beta$ such that $\vert \beta \vert =p-1$, we define
\begin{flalign*}
\alpha_{\beta} (x,y) = \sum_{i=1}^m \theta_i^{2\beta_i+1} \varphi_i (x,y) (u_i(y,t)-u_i(x,t)),
\end{flalign*}
and for each fixed $(x,y) \in \Omega \times \Omega$ we define $G_{(x,y)}: \mathbb{R}^m \to \mathbb{R}$ via 
\begin{flalign*}
G_{(x,y)}(z) = \sum_{\vert \beta \vert = p-1} \binom{p}{\beta} \theta^{\beta^2} z^{\beta} \alpha_{\beta} (x,y).
\end{flalign*}
Then the mean value theorem implies there exists $c(x,y,t,\beta) \in (0,1)$ so that 
\begin{flalign*}
G_{(x,y)} (u(y,t))-G_{(x,y)} (u(x,t)) = \sum_{j=1}^m D_{z_j} G_{(x,y)} (\tilde{u}(x,y,t,\beta))(u_j(y,t)-u_j(x,t)),
\end{flalign*}
where $D_{z_j} G_{(x,y)}$ is the partial of $G$ with respect to $z_j$ and 
\begin{flalign*}
\tilde{u}(x,y,t,\beta) = u(x,t)+c(x,y,t,\beta) (u(y,t)-u(x,t)).
\end{flalign*}

It is important to note that when $p=2$, the expression $D_{z_j}G_{(x,y)}(\tilde{u}(x,y,t,p))$ is a positive constant that only depends on $\beta$. 
Regardless of the value of the integer $p \ge 2$, the integrand in $I$ causes us to investigate
\begin{flalign*}
\sum_{\vert \beta \vert=p-1} \sum_{i,j=1}^m \binom{p}{\beta} \theta^{\beta^2} \beta_j \tilde{u}(x,y,t,\beta)^{\beta-\Vec{e_j}}(z_j-w_j) \theta_i^{2\beta_i+1} \varphi_i(x,y)(z_i-w_i)
\end{flalign*}
for each $w=(w_k),z=(z_k) \in \mathbb{R}_+^m$ where $\Vec{e_j}$ is column $j$ of the $m \times m$ identity matrix. For notational convenience, we define 
\begin{flalign*}
Z_k = z_k-w_k \ \text{for} \ k \in \{1,2,\dots,m\}.
\end{flalign*}

\noindent\textbf{Claim:} 
\begin{flalign}
\sum_{\vert \beta \vert=p-1} \sum_{i,j=1}^m \binom{p}{\beta} \theta^{\beta^2} \beta_j \tilde{u}(x,y,t,\beta)^{\beta-\Vec{e_j}}(z_j-w_j) \theta_i^{2\beta_i+1} \varphi_i(x,y)(z_i-w_i)\nonumber\\  
= \sum_{\vert \beta \vert=p-1} \sum_{i,j=1}^m \binom{p}{\beta} \theta^{\beta^2} \beta_j \tilde{u}(x,y,t,\beta)^{\beta-\Vec{e_j}}Z_j \theta_i^{2\beta_i+1} \varphi_i(x,y)Z_i \label{eq:5.1}\\
= \sum_{\vert \beta \vert=p-2} \binom{p}{\beta} \theta^{\beta^2} \tilde{u}(x,y,t,\beta)^{\beta} \sum_{i,j=1}^m b_{i,j}(x,y) Z_i Z_j \label{eq:5.2}
\end{flalign}
where $B(x,y)=\left(b_{i,j}(x,y)\right)$ is the $m \times m$ symmetric matrix with entries 
\begin{flalign*}
b_{i,j}(x,y) = \varphi(x,y)
\begin{cases}
    \frac{d_i+d_j}{2} \theta_i^{2\beta_i+1}\theta_j^{2\beta_j+1}, &\text{if} \ i \neq j \\
    d_i\theta_i^{4\beta_i+4}, &\text{if} \ i=j
\end{cases}.
\end{flalign*}
We remark that if we define the $m\times m$ matrix $C=\text{diag}\left(\theta_i^{2\beta_i+1}\right)$ then 
$$\varphi C\mathscr{M}C=B.$$
As a result, from \hyperlink{item:(DIFF)}{(DIFF)} and our hypothesis concerning $\mathscr{M}$, if the claim above is true then the matrix $B$ is nonnegative definite. Consequently, if we can prove the claim above, then we can conclude that
$$ (\mathcal{L}_p[u])'(t)\le C_p\left(1+\mathcal{L}_p[u](t)\right),$$
and complete the proof of the Lemma.

To this end, the right hand side of (\ref{eq:5.1}) is a quadratic form in the variable $Z$. As a result, we can find a real symmetric $m \times m$ matrix $M(\theta,p,x,y,t)$ so that the right hand side of (\ref{eq:5.2}) can be written as 
\begin{flalign*}
Z^T M(\theta,p,x,y,t) Z 
\end{flalign*}
where 
\begin{flalign*}
M(\theta,p,x,y,t) = \left( \sum_{\vert \beta \vert=p-2} \binom{p}{\beta} \theta^{\beta^2} \tilde{u}(x,y,t,\beta)^{\beta} b_{i,j}(x,y) \right)
\end{flalign*}
is an $m\times m$ matrix. We start by writing the right hand side of (\ref{eq:5.1}) as following
\begin{equation*}
\begin{aligned}
\sum_{\vert \beta \vert=p-1} \binom{p}{\beta} \theta^{\beta^2} \beta_j \tilde{u}(x,y,t,\beta)^{\beta-\Vec{e_j}} \theta_i^{2\beta_i+1} \varphi_i(x,y)+\sum_{\vert \beta \vert=p-1} \binom{p}{\beta} \theta^{\beta^2} \beta_i \tilde{u}(x,y,t,\beta)^{\beta-\Vec{e_i}} \theta_j^{2\beta_j+1} \varphi_j(x,y)=I_1+I_2
\end{aligned}
\end{equation*}
It is straightforward to see that
\begin{equation*}
\begin{aligned}
\text{I}_1&=\sum_{|\beta|=p-1,\beta_j\ge1}\left(\begin{array}{c}
p\\
\beta
\end{array}\right)\theta^{\beta^{2}}\beta_j \tilde{u}(x,y,t,\beta)^{\beta-\vec{e_j}}\theta_{i}^{2\beta_{i}+1} \varphi_i(x,y)\nonumber\\
&=\sum_{|\beta|=p-1,\beta_j\ge1}\left(\begin{array}{c}
p\\
\beta-\vec{e_j}
\end{array}\right)\theta^{\left(\beta-\vec{e_j}\right)^{2}} \tilde{u}(x,y,t,\beta)^{\beta-\vec{e_j}}\theta_j^{2(\beta_j-1)+1}\theta_{i}^{2\beta_{i}+1} \varphi_i(x,y)\nonumber\\
&=\sum_{|\beta|=p-2}\left(\begin{array}{c}
p\\
\beta
\end{array}\right)\theta^{\beta^{2}} \tilde{u}(x,y,t,\beta)^{\beta}\theta_j^{2\beta_j+1}\theta_{i}^{2\beta_{i}+1} \varphi_i(x,y).
\end{aligned}
\end{equation*}
Similarly,
\begin{equation*}
\text{I}_2=\sum_{|\beta|=p-2}\left(\begin{array}{c}
p\\
\beta
\end{array}\right)\theta^{\beta^{2}} \tilde{u}(x,y,t,\beta)^{\beta}\theta_j^{2\beta_j+1}\theta_{i}^{2\beta_{i}+1} \varphi_j(x,y).
\end{equation*}
As a result, for $i\ne j$ we have
\[
M_{i,j}(\theta,p,x,y,t)=\sum_{|\beta|=p-2}\left(\begin{array}{c}
p\\
\beta
\end{array}\right)\theta^{\beta^{2}} \tilde{u}(x,y,t,\beta)^{\beta}\theta_j^{2\beta_j+1}\theta_{i}^{2\beta_{i}+1} \frac{\phi_i(x,y)+\phi_j(x,y)}{2}.
\]
Now we consider the case when $i=j\in\{1,...,m\}$. If we denote the coefficient on $(Z_i)^2$ on the right-hand side of (\ref{eq:5.2}) as $I_3$ then
\begin{align*}
I_3&=\sum_{|\beta|=p-1,\beta_i\ge1}\left(\begin{array}{c}
p\\
\beta
\end{array}\right)\theta^{\beta^{2}}\beta_i \tilde{u}(x,y,t,\beta)^{\beta-\vec{e_i}}\theta_{i}^{2\beta_{i}+1} \phi_i(x,y)\nonumber\\
&=\sum_{|\beta|=p-1,\beta_i\ge1}\left(\begin{array}{c}
p\\
\beta-\vec{e_i}
\end{array}\right)\theta^{\left(\beta-\vec{e_i}\right)^{2}} \theta_i^{2(\beta_i-1)+1}\tilde{u}(x,y,t,\beta)^{\beta-\vec{e_i}}\theta_{i}^{2(\beta_{i}-1)+3} \phi_i(x,y)\nonumber\\
&=\sum_{|\beta|=p-2}\left(\begin{array}{c}
p\\
\beta
\end{array}\right)\theta^{\beta^{2}} \theta_i^{2\beta_i+1}\tilde{u}(x,y,t,\beta)^{\beta}\theta_{i}^{2\beta_{i}+3} \phi_i(x,y)\nonumber\\
&=\sum_{|\beta|=p-2}\left(\begin{array}{c}
p\\
\beta
\end{array}\right)\theta^{\beta^{2}} \theta_i^{4\beta_i+4}\tilde{u}(x,y,t,\beta)^{\beta} \phi_i(x,y).
\end{align*}
So, $I_1,I_2$ and $I_3$ prove the claim. Therefore, from above, we see that
\begin{equation*}
I  = -\frac{1}{2} \int_{\Omega \times \Omega} Z^T M(\theta, p,x,y,t) Z \ d(x,y)\le 0.
\end{equation*}
As a result,
$$ (\mathcal{L}_p[u])'(t)\le C_p(t)\left(1+\mathcal{L}_p[u](t)\right),$$
implying there exists a continuous function $K_p\in C(\mathbb{R}_+,\mathbb{R}_+)$ so that 
$\mathcal{L}_p[u](t)\le K_p(t)$ for all $t\ge0$. Therefore, the result follows from (\ref{equivnorm}). \hfill \(\square\)\\

\subsection{Proof of Theorem \ref{DiffLimit}}

Denote $u^{(j)}=\left(u_i^{(j)}\right)$ as the global component-wise nonnegative solution to (\ref{eq:2.17}). We denote $\varphi_j(x,y)=j^{n+2}\psi(j|x-y|)$. Let $T>0$. Then from lemma \ref{Lpbdd}, if $1\le p<\infty$ there exists $C_{p,T}>0$ so that $\|u_i^{(j)}(\cdot,t)\|_{p,\Omega}\le C_{p,T}$ for all $j\in\mathbb{N}$ and $i=1,...,m$. We define
\begin{align}
Y_{j}[z] := \int_{\Omega \times \Omega} \varphi_j(x,y) [z(x)-z(y)]^2 \ d(x,y), \ \ \text{for} \ z \in L^2(\Omega). \label{eq:5.7}
\end{align}
\begin{lemma} \label{lem:5.2} If $T>0$, $j \in\mathbb{N}$, and $1\le p<\infty$ then there exists $C_{p,T}>0$ so that
\begin{subequations}
\begin{align}
     \sum_{i=1}^m \Vert u_{i}^{(j)}(,\cdot,t) \Vert^p_{p,\Omega} & \le C_{p,T}\, \ \text{for all }0\le t\le T, \label{eq:5.8a} \\
     \text{and} \ \int_0^t \sum_{i=1}^m Y_j[u_{i}^{(j)}(\cdot,s)] \ ds & \le C_{p,T} \, \ \text{for all }0\le t\le T. \label{eq:5.8b}
\end{align}
\end{subequations} 
\end{lemma}

\textbf{Proof:} Note that (\ref{eq:5.8a}) was obtained in lemma \ref{Lpbdd}. Let $0\le t\le T$ and $j \in \mathbb{N}$. If we multiply the PDE for $u_i^{(j)}$ (\ref{eq:2.17}) by $u_i^{(j)}$  and integrate over $\Omega$ then
\begin{flalign}
\frac{1}{2} \frac{d}{dt} \int_{\Omega} u_{i}^{(j)}(x,t)^2 \ dx & =  d_i \int_{\Omega \times \Omega} \varphi_j(x,y) (u_{i}^{(j)}(y,t)-u_{i}^{(j)}(x,t)) u_{i}^{(j)}(x,t) \ dx \ dy + \int_{\Omega} f_{i}(u^{(j)}(x,t)) u_{i}^{(j)}(x,t) \ dx \nonumber \\
& = A+B, \label{eq:5.9}
\end{flalign}
where 
\begin{flalign*}
A & =  d_i \int_{\Omega \times \Omega} \varphi_j(x,y) (u_{i}^{(j)}(y,t)-u_{i}^{(j)}(x,t)) u_{i}^{(j)}(x,t) \ dxdy  = -  \frac{d_i}{2} \int_{\Omega \times \Omega} \varphi_j(x,y) (u_{i}^{(j)}(y,t)-u_{i}^{(j)}(x,t))^2 \ d(x,y) \\
& = -\sum_{i=1}^m \frac{d_i}{2} Y_j[u_{i}^{(j)}(t)]
\end{flalign*}
from lemma \ref{lem:3.1}. Also, from (Lp Continuity) we have
$$B \le \int_{\Omega} L(t)\left(\sum_{i=1}^m u_i^{(j)}(x,t)\,+1\right)^r u_i^{(j)}(x,t)\ dx. $$
Consequently, from (\ref{eq:5.8a}) there is a constant $K_{p,T}>0$ so that $B\le K_{p,T}$. Inserting this information in (\ref{eq:5.9}) gives 
$$\frac{1}{2} \frac{d}{dt} \int_{\Omega} u_{i}^{(j)}(x,t)^2 \ dx+\sum_{i=1}^m \frac{d_i}{2} Y_j[u_{i}^{(j)}(t)]\le K_{p,T}.$$
If we integrate this with respect to $t$ from $0$ to $T$ we get our result. \hfill$\square$
 
\begin{lemma} \label{lem:5.3} There exists $C_T>0$ such that
\begin{flalign}
   \int_0^t \big\Vert \frac{\partial}{\partial t} u_{i}^{(j)}(\cdot,s) \big \Vert^2_{\left(W_{n+1}^{(1)}(\Omega)\right)'} \ ds \le C_T \label{eq:5.10}
\end{flalign}
for all $0\le t\le T$, $i=1,...,m$ and $j \in \mathbb{N}$.
\end{lemma}

\textbf{Proof:} Let $\zeta \in W_{n+1}^{(1)}(\Omega)$, $t>0$, and $j \in \mathbb{N}$. Then we have,
\begin{flalign}
   \left| \int_{\Omega} \zeta(x) \frac{\partial}{\partial t}  u_{i}^{(j)}(x,t) \ dx \right| & \le \big \vert \int_{\Omega} \zeta(x) d_i \int_{\Omega} \varphi_j(x,y) (u_{i}^{(j)}(y,t)-u_{i}^{(j)}(x,t)) \ dy \ dx\big \vert  + \big \vert\int_{\Omega} \zeta(x) f_{i}(t,u^{(j)}(x,t)) \ dx \big \vert \nonumber \\  
   & \le J_1+J_2 \label{eq:5.11}
 \end{flalign}  
where from lemma \ref{lem:3.1} we can write,
\begin{flalign*}
   J_1 & = \big \vert d_i \int_{\Omega \times \Omega} \zeta(x) \varphi_j(x,y) (u_{i}^{(j)}(y,t)-u_{i}^{(j)}(x,t)) \ dx \ dy \big \vert  =  \frac{d_i}{2} \big \vert \int_{\Omega \times \Omega} \varphi_j(x,y) (\zeta(x)-\zeta(y)) (u_{i}^{(j)}(y,t)-u_{i}^{(j)}(x,t)) \ dx \ dy \big \vert \\
& \le \frac{d_i}{2} \big [ \big( \int_{\Omega \times \Omega} \varphi_j(x,y) (\zeta(x)-\zeta(y))^2 \ dx \ dy \big)^{\frac{1}{2}} \big( \int_{\Omega \times \Omega} \varphi_j(x,y) ((u_{i}^{(j)}(y,t)-u_{i}^{(j)}(x,t))^2 \ dx \ dy \big)^{\frac{1}{2}} \big] 
\end{flalign*}
from Holder's inequality. Since, the support of $\varphi_j(x,y)$ is a subset of $\big\{(x,y) \in \mathbb{R}^{2n} : \vert x-y \vert \le \frac{C_6}{j} \big \}$, we can follow the proof of Lemma 4.2 in \cite{laurencot2023nonlocal} and use Lemma \ref{lem:5.2} and \hyperlink{item:(DIFF2)}{(DIFF2)} to find $C_{7,T}>0$ so that 
\begin{equation*}
J_1 \le C_{7,T} \Vert \zeta \Vert_{n+1,\Omega}^{(1)}\big(\sum_{i=1}^m Y_j[u_{i}^{(j)}(t)]\big)^{\frac{1}{2}}\,\text{for }0\le t\le T.
\end{equation*}
In addition, there exists $C_{8,T}>0$ so that
\begin{flalign*}
J_2 & = \big \vert \int_{\Omega} \zeta(x)  f_{i}(t,u_{i}^{(j)}(x,t)) \ dx \big \vert  \le \big \vert \int_{\Omega} \zeta(x) L(t) \big(\sum_{i=1}^m u_{i}^{(j)}(x,t) + 1 \big)^r \ dx \big \vert \le C_{8,T} \Vert \zeta \Vert_{n+1,\Omega}^{(1)}\,\text{for all }0\le t\le T
\end{flalign*}
from \hyperlink{item:(Lp Continuity)}{(Lp Continuity)}, Lemma \ref{lem:5.2} and the fact that $W_{n+1}^{(1)}(\Omega)$ is continuously embedded in $L^{\infty}(\Omega)$.
 Therefore, $(\ref{eq:5.11})$ gives
$$\big \vert \int_{\Omega} \zeta(x) \frac{\partial}{\partial t}  u_{i}^{(j)}(x,t) \ dx \big \vert  \le C_{7,T} \Vert \zeta \Vert_{n+1}^{(1)}\big(Y_j[u_{i}^{(j)}(t)]\big)^{\frac{1}{2}}+ C_{8,T} \Vert \zeta \Vert_{n+1,\Omega}^{(1)}.$$
Integrating with respect to $t$, and applying duality and lemma \ref{lem:5.2} gives the result. \hfill $\square$\\

The following result is given in \cite{laurencot2023nonlocal}.

\begin{lemma} \label{lem:5.4} If $T>0$ then the sequence $\{u^{(j)}\}$ is relatively compact in $L^2(Q_{0,T},\mathbb{R}^m)$ and their cluster points belong to $L^2((0,T),H^1(\Omega,\mathbb{R}^m))$.
\end{lemma}

\noindent\textbf{Proof of Theorem \ref{DiffLimit}:} 
From lemma \ref{lem:5.4} there exists a subsequence $\{u^{(j_k)}\}$ and a component-wise nonnegative function $U$ such that
$$U\in C([0,T],\left(W_{n+1}^{(1)}(\Omega)\right)'\cap L_\infty\left((0,T),L^2(\Omega)\cap L^2\left((0,T),H^1(\Omega)\right)\right)$$
and
\begin{equation}\label{limit1}
\lim_{k\to\infty}\sup_{0\le t\le T}\|\left(u_i^{(j_k)}-U_i\right)(\cdot,t)\|_{(W_{n+1}^{(1)}(\Omega))'}=0\,\text{for all }i=1,...,m,
\end{equation}
\begin{equation}\label{limit2}
\lim_{k\to\infty}\int_0^T\|\left(u_i^{(j_k)}-U_i\right)(\cdot,t)\|_{2,\Omega}^2=0\,\text{for all }i=1,...,m,
\end{equation}
and
\begin{equation}\label{limit3}
\lim_{k\to\infty}\big\vert \left(u_i^{(j_k)}-U_i\right)(x,t)\big\vert=0\,\text{for all }i=1,...,m\,\text{and a.e. }(x,t)\in Q_{0,T}.
\end{equation}
Furthermore, since $u^{(j_k)}\to U$ in $L^2(Q_{0,T})$ and $u^{(j_k)}$ is bounded in $L^p(\Omega)$ independent of $k\in\mathbb{N}$ for every $1\le p<\infty$, it follows that $u^{(j_k)}\to U$ in $L^p(Q_{0,T})$ for every $1\le p<\infty$. Consequently, from \hyperlink{item:(Lp Continuity)}{(Lp Continuity)}, we can conclude that $f_i(u^{(j_k)})\to f_i(U)$ in $L^2(Q_{0,T})$ for each $i=1,...,m$. The result follows by applying Lemma \ref{prop:3.1} and a diagonalization argument. \hfill $\square$

\subsection{Proof of Theorem \ref{DiffLimit2}}

Throughout this subsection, we assume $\lambda\in\mathbb{N}$ such that $\lambda\ge 2$ and $\Gamma_j$ is replaced in (\ref{eq:2.19}) with  
$$\Gamma_j(u_j)(x,t)=\int_\Omega \lambda^{n+2}\psi(\lambda|x-y|)(u_j(y,t)-u_j(x,t))dy.$$ As a result, we denote the solution to (\ref{eq:2.19}) by $u^{(\lambda)}$.
Theorem \ref{globalmixed} implies (\ref{eq:2.19}) has a global component-wise nonnegative solution $u^{(\lambda)}$. In order to prove Theorem \ref{DiffLimit2} we need to obtain the results in the proof of Theorem \ref{globalmixed} in such a way that the $L^p$ bounds that are obtained in that setting are independent of $\lambda^{n+2}\psi(\lambda|x-y|)$.  

To this end, note that \hyperlink{item:(INT-SUM)}{(INT-SUM)}, the symmetry of $\psi(\lambda|x-y|)$ and the homogeneous Neumann boundary conditions on $u_j$ for $j\in\{m_1+1,...,m\}$  imply
\begin{equation*}
\frac{\partial}{\partial t}\sum_{i=1}^m\int_\Omega a_{m,i}u_i^{(\lambda)}(x,t)dx\le L\left(\sum_{i=1}^m\int_\Omega u_i^{(\lambda)}(x,t)+1\right)\le L\alpha\left(\sum_{i=1}^m\int_\Omega a_{m,i}u_i^{(\lambda)}(x,t)+1\right)
\end{equation*}
where $\alpha=\min_{i\in\{1,...,m\}}a_{m,i}>0$. Consequently, it follows that there exists a nondecreasing function $g\in C(\mathbb{R}_+,\mathbb{R}_+)$ that is independent of $\lambda^{n+2}\psi(\lambda|x-y|)$ such that 
\begin{equation}\label{L1-diffbound}
\|u_i^{(\lambda)}(\cdot,t)\|_{1,\Omega}\le g(t)\quad\text{for all }i\in\{1,...,m\}\,\text{and }t\ge0.
\end{equation}

Now define $U^{(\lambda)}=(u_i^{(\lambda)})_{i=1}^{m_1}$ and let $\tilde{A}=(a_{i,j})_{i,j\le m_1}$. Then proceeding as in the proof of Lemma \ref{Lpbdd}, we find that if $2\le p'\in\mathbb{N}$ then
$$ (\mathcal{L}_{p'}[U^{(\lambda)}])'(t)\le C_{p'}(t)\left(1+\mathcal{L}_{p'}[U^{(\lambda)}](t)+\sum_{i=m_1+1}^m\int_\Omega u_i(x,t)^{p'}dx\right),$$
with $C_{p'}$ independent of $\lambda^{n+2}\psi(\lambda|x-y|)$. It follows that there exists $M_{p'}\in C(\mathbb{R}_+,\mathbb{R}_+)$ which is nondecreasing and $M_{p'}(0)=0$ such that
\begin{flalign}
\sum_{i=1}^{m_1} \|u_{i}^{(\lambda)} \| _{p',Q_{\tau,T}} & \le M_{p'}(T-\tau)\left(\sum_{i=m_1+1}^{m} \|u_i^{(\lambda)}(\cdot,\tau)\|_{p',\Omega}+1 + \sum_{i=m_1+1}^{m} \|u_i^{(\lambda)}\|_{p',Q_{\tau,T}}\right). \label{Lp-diffbound}
\end{flalign}

Now we proceed as in the proof of Theorem \ref{globalmixed}. To this end, let $0\le\tau<T$, $j\in\{m_1+1,\dots,m\}$, $1<p<<\frac{n+2}{2}$ (with the proximity to $1$ to be specified below) and $\theta \in L^p(Q_{\tau,T})$ such that $\theta\ge0$ and $\|\theta\|_{p,Q_{\tau,T}}=1$. Let $\phi_j$  be the unique nonnegative solution in $W_p^{(2,1)}(Q_{\tau,T})$ solving
\begin{equation} \label{eq:phidiff}
\left\{\begin{aligned}
    \frac{\partial}{\partial t} \phi_{j} & = - d_{j} \Delta \phi_{j} - \theta, &&Q_{\tau,T},\\
    \frac{\partial}{\partial \eta} \phi_{j} & = 0, &&\partial\Omega \times (\tau,T),\\
    \phi_{j} & = 0,&& \Omega \times \{T\}.
\end{aligned}\right.
\end{equation}
From Lemma \ref{lem:3.3} there exists $C_{p,d_j,T-\tau}>0$ so that $$\|\phi_j(\cdot,\tau)\|_{p,\Omega},\,\|\phi_j\|_{p,Q_{\tau,T}}^{(2,1)}\le C_{p,d_j,T-\tau},$$ 
along with additional embedding results. We recall that $C_{p,d_j,T-\tau}$ can be chosen to be nondecreasing in $T-\tau$, so we can choose $C_{p,T-\tau}=\max_{\{j=m_1+1,...,m\}}C_{p,d_j,T-\tau}$, and without loss of generality assume $\|\Delta \phi_j\|_{p,Q_{\tau,T}}\le C_{p,T-\tau}$. Finally, let $p'=\frac{p}{p-1}$ and recall the function $M_{p'}$ from (\ref{Lp-diffbound}). Note that \hyperlink{item:(INT-SUM)}{(INT-SUM)} implies $a_{i,j}>0$ for all $i\ge m_1$ and $j\le i$. Let $k\in\{m_1+1,\dots,m\}$ and set $j=k$ in (\ref{eq:phidiff}). In addition, let $\mathbb{I}_k=\{j\in\mathbb{N}\, |\,j\ge m_1+1\,\text{and } j<k\}$. Then
\begin{equation*}
\begin{aligned}
\int_\tau^T\int_\Omega \sum_{i=1}^k a_{k,i}u_i^{(\lambda)}\theta dxdt&=\int_\tau^T\int_\Omega \sum_{i=1}^k a_{k,i}u_i^{(\lambda)}(-d_k\Delta\phi_k-\partial_t\phi_k)dxdt\\
&=-\int_\tau^T\int_\Omega \sum_{i=1}^k a_{k,i}u_i^{(\lambda)} d_k\Delta\phi_kdxdt+\int_\Omega \sum_{i=1}^k a_{k,i}u_i^{(\lambda)}(x,\tau)\phi_k(x,\tau)dx+\int_\tau^T\int_\Omega \phi_k\sum_{i=1}^k a_{k,i}\partial_tu_i^{(\lambda)}dxdt\\
&=-\int_\tau^T\int_\Omega \sum_{i=1}^{m_1} a_{k,i}u_i^{(\lambda)} d_k\Delta\phi_kdxdt+\int_\Omega \sum_{i=1}^k a_{k,i}u_i^{(\lambda)}(x,\tau)\phi_k(x,\tau)dx\\
&+\int_\tau^T\int_\Omega \phi_k(x,t)\sum_{i=1}^{m_1} a_{k,i}\int_\Omega d_i\lambda^{n+2}\psi(\lambda|x-y|)(u_i^{(\lambda)}(y,t)-u_i^{(\lambda)}(x,t))dydxdt\\
&+\int_\tau^T\int_\Omega \sum_{i\in \mathbb{I}_k} a_{k,i} (d_i-d_k)u_i^{(\lambda)}\Delta\phi_k dxdt+\int_\tau^T\int_\Omega \phi_k L\left(\sum_{i=1}^{m} u_i^{(\lambda)}+1\right)dxdt\\
&=I_1+I_2+I_3+I_4+I_5.
\end{aligned}
\end{equation*}
From the bounds in (\ref{L1-diffbound}) and (\ref{Lp-diffbound}) and the arguments in the proof of Theorem \ref{globalmixed}, it is clear that if we can obtain an estimate for $I_3$ that is independent of $\lambda$ by choosing $\tau$ sufficiently close to $T$, then we can obtain estimates for $\|u_i\|_{p',Q_{\tau,T}}$ for all $i\in\{1,...,m\}$ that are independent of $\lambda$, and only depend on $T-\tau$, $p'$ and parameters in \hyperlink{item:(INT-SUM)}{(INT-SUM)}. Furthermore, once this is accomplished we the result will follow. 

To this end, let $$A_{max}=\max_{i,j\in\{1,...,m\}}a_{i,j}\,>0.$$
We note that $Q_{\tau,T}$ is a Lipschitz domain. As a result, there is an extension operator $E:W_p^{(2,1)}(Q_{\tau,T})\to W_p^{(2,1)}(\mathbb{R}^n\times(\tau,T))$ and a constant $C_{p,E}>0$ such that $E(g)=g$ on $Q_{\tau,T}$ and $\|E(g)(\cdot,t)\|_{p,\mathbb{R}^n}^{(2)}\le C_{p,E}\|g(\cdot,t)\|_{p,\Omega}^{(2)}$. For each $i\in\{1,...,m_1\}$ extend $u_i^{(\lambda)}(\cdot,t)$ outside $\Omega$ as the zero function. Define $$W_{\lambda}(z)=\lambda^{n+2}\psi(\lambda|z|).$$
Note that 
$$\int_{\mathbb{R}^n}W_\lambda(z)|z|^2dz=\int_{\mathbb{R}^n}\lambda^{n+2}\psi(\lambda|z|)|z|^2dz=\int_{\mathbb{R}^n}\psi(|u|)|u|^2du=M$$
from the substitution $u=\lambda z$ and \hyperlink{item:(DIFF2)}{(DIFF2)}.
Then Lemma \ref{lem:2nddiff} implies
\begin{equation*}
\begin{aligned}
\int_{\mathbb{R}^n} \int_{\mathbb{R}^n} W_\lambda(x-y)E(\phi_k)(x,t)(u_i^{(\lambda)}(y,t)-u_i^{(\lambda)}(x,t))dydx&\le \frac{M}{2}\|E(\phi_k)(\cdot,t)\|_{W_p^{(2)}(\mathbb{R}^n)}\|u_i^{(\lambda)}(\cdot,t)\|_{p',\Omega}\\
&\le \frac{MC_{p,E}}{2}\|\phi_k(\cdot,t)\|_{W_p^{(2)}(\Omega)}\|u_i^{(\lambda)}(\cdot,t)\|_{p',\Omega}.
\end{aligned}
\end{equation*}
As a result, 
\begin{equation}\label{I3ind}
\begin{aligned}
I_3&\le d_{max}A_{max}\frac{MC_{p,E}}{2}\int_\tau^T \|\phi_k(\cdot,t)\|_{W_p^{(2)}(\Omega)}\|u_i^{(\lambda)}(\cdot,t)\|_{p',\Omega}dt\\
&\le K_p C_{p,(T-\tau)}M_{p'}(T-\tau)\left(\sum_{i=m_1+1}^{m} \|u_i^{(\lambda)}(\cdot,\tau)\|_{p',\Omega}+1 + \sum_{i=m_1+1}^{m} \|u_i^{(\lambda)}\|_{p',Q_{\tau,T}}\right).
\end{aligned}
\end{equation}
Note that (\ref{I3ind}) does not depend on $\lambda$, and as a result, we can proceed as in the proof of Theorem \ref{globalmixed} to show that if $p'$ is sufficiently large then there exists $\beta_{p'}>0$ so that if $\tau\ge 0$ then we obtain bounds for $\|u_i\|_{p',Q_{\tau,\tau+\beta_{p'}}}$ for all $i\in\{1,...,m\}$ that are independent of $\lambda$. Consequently, given any $T>0$ we can obtain bounds for $\|u_i\|_{p',Q_{0,T}}$ for all $i\in\{1,...,m\}$ that are independent of $\lambda$. The remainder of the proof follows similar to the proof of Theorem \ref{DiffLimit}.

\section{Examples and Numerical Simulations}

In this section, we provide some examples of nonlocal reaction-diffusion systems which satisfy the theorems stated in section \ref{Main Results}. Then we perform some numeric calculations to illustrate both the diffusive limit and the difference in behavior of solutions to systems with nonlocal diffusion with solutions to systems with local diffusion. 

\subsection{Examples}

Consider the following reversible chemical reactions:
\[
\text{S}_1 \xrightleftharpoons[k_1(x,t)]{k_1(x,t)} 2\text{S}_2,
\quad
\text{S}_2 \xrightleftharpoons[k_2(x,t)]{k_2(x,t)} 2\text{S}_3.
\]
The associated reaction-diffusion system was studied in \cite{desvillettes2024trend} and is listed below.
\[
\left\{
\begin{aligned}
\partial_t u_1 &= \nabla \cdot (d_1(x,t) \nabla u_1) + k_1(x,t) \left( u_2^2 - u_1 \right), &&\Omega \times \mathbb{R}_+, \\
\partial_t u_2 &= \nabla \cdot (d_2(x,t) \nabla u_2) -2k_1(x,t) \left( u_2^2 - u_1 \right) + k_2(x,t) \left( u_3^2 - u_2 \right), && \Omega \times \mathbb{R}_+, \\
\partial_t u_3 &= \nabla \cdot (d_3(x,t) \nabla u_3) -2k_2(x,t) \left( u_3^2 - u_2 \right), && \Omega \times \mathbb{R}_+, \\
\frac{\partial u_1}{\partial \eta} &= \frac{\partial u_2}{\partial \eta} = \frac{\partial u_3}{\partial \eta} = 0, && \partial\Omega \times \mathbb{R}_+, \\
u_1(\cdot, 0) &= u_{1,0},\quad u_2(\cdot, 0) = u_{2,0},\quad u_3(\cdot, 0) = u_{3,0}, && \Omega.
\end{aligned}
\right.
\]
where $u_i$ is the concentration density of $S_i$, the functions $d_i$ and $k_i$ are positive, and the initial data $u_{i,0}$ is bounded and nonnegative. If we rename the components \( v_1 = u_3 \), \( v_2 = u_2 \), and \( v_3 = u_1 \), replace the local diffusion with our  nonlocal diffusion operator $\Gamma_i(u)$, and treat the \( k_i \) terms as constants, then for $i=1,2,3$, we obtain a model of the form:
\[
\left\{
\begin{aligned}
&\partial_t v_i = d_i \Gamma(v_i) + f_i(v), &&\text{in } \Omega \times \mathbb{R}_+, \\
&v_i(x,0) = v_{0,i}(x), &&\text{in } \Omega.
\end{aligned}
\right.
\]
with the reaction vector field
\[
f(v) = \begin{pmatrix}
2k_2 (v_2 - v_1^2) \\
-2k_1 (v_2^2 - v_3) + k_2(v_1^2 - v_2) \\
k_1 (v_2^2 - v_3)
\end{pmatrix}.
\]
We note that the reaction vector field is quasi positive, and
\[
f_1(v) + 2f_2(v) + 4f_3(v) = 0,
\]
for all $v \in \mathbb{R}_+^3$, and consequently, we obtain global existence of component-wise nonnegative solutions with uniform sup norm bounds from Theorem \ref{global}. 

If we modify the system above so that \( v_1 \) and \( v_2 \) diffuse nonlocally and \( v_3 \) diffuses locally (with homogeneous Neumann boundary conditions), then the system becomes:
\[
\begin{cases}
\partial_t v_i = d_i \Gamma(v_i) + f_i(v), & x \in \Omega, \; t \ge 0,\; i = 1,2, \\
\partial_t v_3 = d_3 \Delta v_3 + f_3(v), & x \in \Omega, \; t > 0, \\
\frac{\partial v_3}{\partial \eta} = 0, & x \in \partial \Omega, \; t > 0, \\
v_i(x,0) = v_{0,i}(x), & x \in \Omega,\, i=1,2,3.
\end{cases}
\]
We observe that
\[
f_1(v) + f_2(v) = 2k_1 v_3 + k_2 v_2,
\quad \text{and} \quad
f_1(v) + 2f_2(v) + 4f_3(v) = 0,
\]
for all $v \in \mathbb{R}_+^3$, and conseuqently, global existence follows from Theorem \ref{globalmixed}. 

Finally, we can perform the diffusive limit on the full nonlocal model to derive convergence toward the corresponding classical reaction-diffusion system (with Neumann boundary conditions), as stated in Theorem \ref{DiffLimit}
since the system above satisfies the linear intermediate sum conditions:
\[
f_1(v) \le 2k_2 v_2,
\quad
f_1(v) + f_2(v) \le 2k_1 v_3 + k_2 v_2,
\quad
f_1(v) + 2f_2(v) + 4f_3(v) = 0, \quad \text{for all } v \in \mathbb{R}_+^3.
\]
Moreover, the result in \cite{morgan1990boundedness} implies that solutions to the limiting system with local diffusion and homogeneous Neumann boundary conditions are uniformly bounded in the sup-norm because the reaction vector field satisfies a linear intermediate sum condition and
\[
f_1(v) + 2f_2(v) + 4f_3(v) = 0
\]
 for all $v \in \mathbb{R}_+^3$. 

\bigskip Now, let's consider a rumor model  studied in \cite{chen2020rumor} with no consideration of diffusion. If we add nonlocal diffusion to their model, then we obtain the system given by
\[
\left\{
\begin{aligned}
&\partial_t v_i = d_i \Gamma(v_i) + f_i(v), &&\text{in } \Omega \times \mathbb{R}_+, \ i=1,2,3,4,5 \\
&v_i(x,0) = v_{0,i}(x), &&\text{in } \Omega,\,i=1,...,5.
\end{aligned}
\right.
\]
with the reaction vector field
\[
f(v) = 
\begin{pmatrix}
-\bar{k} \, v_1 \, v_4 \left( \gamma \alpha \lambda \mu + \gamma(1 - \gamma) \alpha \lambda \right) \\
-\bar{k} \, v_2 \, v_4 \gamma \alpha \lambda \\
\bar{k} \, v_1 \, v_4 \gamma(1 - \gamma)\alpha \lambda 
- \bar{k} \, v_4 \, v_3 \theta 
- \bar{k} \, v_5 \, v_3 \phi \\
\bar{k} \, v_4 (\mu v_1 + v_2) \gamma \alpha \lambda 
+ \bar{k} \, v_4 v_3 \theta 
- \bar{k} \, v_4 (v_5 + v_4 + v_3) \eta_1 
- v_4 \eta_2 \\
\bar{k} \, v_4 (v_5 + v_4 + v_3) \eta_1 
+ v_4 \eta_2 
+ \bar{k} \, v_5 v_3 \phi
\end{pmatrix}, \ \ \ \forall \ v \in \mathbb{R}_+^5.
\]
$0<\gamma,\alpha,\mu \le 1$ and $0 \le \lambda,\theta,\phi,\eta_1,\eta_2 \le 1$, and the components $u_1$, $u_2$, $u_3$, $u_4$ and $u_5$ represent the population densities of the ``Steady Ignorant'' population (people who do not know the rumor, and if they hear the rumor, they prefer to contemplate it and seek confirmation before making decisions), the ``Radical Ignorant'' population (people who do not know the rumor, and if they hear the rumor, they are most likely to believe it and spread it without contemplating it or seeking confirmation), the ``Exposed'' population (people who know the rumor but hesitate to believe it and do not spread it), the ``Spreader'' population (people who spread the rumor), and the ``Stifler'' population (people who know the rumor but never spread it or stop spreading it), respectively. We observe that $f$ is quasi-positive and 
\[
f_1(v) + f_2(v) + f_3(v) + f_4(v) + f_5(v) = 0,
\]
for all $v \in \mathbb{R}_+^5$, which yields global existence of component-wise nonnegative solutions with uniform sup norm bounds from Theorem \ref{global}. Furthermore, it is easy to see that $f$ satisfies a linear intermediate sum condition, and consequently, the diffusive limit result in Theorem \ref{DiffLimit} can be obtained.

If we consider a scenario where \( v_1 \), \(v_2 \) and \( v_3 \) diffuse nonlocally and \( v_4 \) and \(v_5 \) diffuse locally (with homogeneous Neumann boundary conditions), then the system becomes:
\[
\begin{cases}
\partial_t v_i = d_i \Gamma(v_i) + f_i(v), & x \in \Omega, \; t > 0,\; i = 1,2,3 \\
\partial_t v_j = d_j \Delta v_j + f_j(v), & x \in \Omega, \; t > 0, \; j = 4,5 \\
\frac{\partial v_j}{\partial \eta} = 0, & x \in \partial \Omega, \; t > 0, \\
v_k(x,0) = v_{0,k}(x), & x \in \Omega, \; k=1,...,5,
\end{cases}
\]
with the same reaction vector field, as noted above, the reaction vector field satisfies a linear intermediate sum condition, and this again implies global existence from Theorem \ref{globalmixed}, along with the diffusive limit result in Theorem \ref{DiffLimit2}.

Lastly, to perform the diffusive limit on the full nonlocal model to derive convergence toward the corresponding classical reaction-diffusion system (with Neumann boundary conditions), as stated in Theorem \ref{DiffLimit}, note that the system above satisfies the linear intermediate sum conditions for all $v \in \mathbb{R}_+^5$, and
\[
f_1(v) + f_2(v) + f_3(v) + f_4(v) + f_5(v) = 0,
\]
for all $v \in \mathbb{R}_+^5$. Therefore, the results in \cite{morgan1990boundedness} imply that solutions to the limiting reaction-diffusion system with local diffusion and homogeneous Neumann boundary conditions are uniformly bounded in the sup-norm.

\subsection{Numerical Simulations}

Here we present some numerical approximations to support our results and illustrate the difference between systems with local diffusion and systems with nonlocal diffusion. We begin by approximating the solutions of our main system 
(\ref{eq:2.11}) in $2$-D settings (for example Gray-Scott Model) using the Finite Difference Method with the Forward Euler Method to calculate the convolution operator defining its integral. Before that, we represent the approximation of local/classical reaction-diffusion systems in $2$-D. In this case, we consider the same Gray-Scott model. Next, we develop the nonlocal model to numerically illustrate Theorem \ref{DiffLimit}. Finally, we use the Method of Lines to present numerically approximated solutions of systems $(\ref{eq:2.11})$ with \hyperlink{item:(INT-SUM2)}{(INT-SUM2)} and $(\ref{eq:2.19})$ with  \hyperlink{item:(INT-SUM2)}{(INT-SUM2)} to show their visual differences.

In all the examples below, we assume $\Omega=(0,2) \times (0,1) \subset \mathbb{R}^2$ which is a rectangle of length, $L=2$m and width, $W=1$m. We discretize $L$ into $100$ slices and $W$ into $50$ slices to create $5000$ mesh-points $(X,Y)$ in our domain. Our method is stable for reaction-diffusion systems since $d_i \Delta t\left(\frac{1}{(\Delta x)^2}+\frac{1}{ (\Delta y)^2}\right) \le 0.5$ for $i=1,2$.

\subsection{Numerical Approximation of Solutions of Local vs Nonlocal RDEs}
First we consider the local or classical Gray-Scott model below
\begin{equation}\label{eq:6.1}
\left\{
\begin{aligned} 
    \partial_t u_{1} & = d_1 \Delta u_1 + f_1(u_1,u_2), && (x,t) \in Q_{0,\infty},\\
    \partial_t u_{2} & = d_2 \Delta u_2 + f_2(u_1,u_2), && (x,t) \in Q_{0,\infty},\\
    \frac{\partial u_1}{\partial \eta} & = \frac{\partial u_2}{\partial \eta} = 0, && (x,t) \in \partial \Omega \times (0,\infty),\\
    (u_1,u_2) (x,0) & = (u_{1,0},u_{2,0}), && x \in \Omega,
\end{aligned}  
\right. 
\end{equation}
where $f_1(u_1,u_2)=a(1-u_1)-u_1u_2^2$ and $f_2(u_1,u_2)=-(a+b)u_2 + u_1u_2^2$ with feed rate $a>0$ and kill rate $b>0$, and $d_1,d_2>0$. We discretize the system to approximate the Laplacian and homogeneous Neumann boundary conditions, and use MATLAB to get figure $(\ref{fig1: Solutions of local Gray-Scott model for $a=0.25$, $b=0.080$, $d_1 = 0.1$, $d_2 = 0.01$.})$ using above information.

Now, using $(\ref{eq:2.11})$ for $i=1,2$, the nonlocal version of Gray-Scott model is,
\begin{equation}\label{eq:6.2}
\left\{
\begin{aligned} 
    \partial_t u_{1} & = d_1 \Gamma_1 u_1 + f_1(u_1,u_2), && (x,t) \in Q_{0,\infty}\\
    \partial_t u_{2} & = d_2 \Gamma_2 u_2 + f_2(u_1,u_2), &&(x,t) \in Q_{0,\infty}\\
    (u_1,u_2) (x,0) & = (u_{1,0},u_{2,0}),  && x \in \Omega.
\end{aligned}  
\right. 
\end{equation}
We consider the following Gaussian kernel with standard deviation, $\epsilon$ for the nonlocal operator
\begin{flalign*}
\varphi(x,y) =\exp\bigl(-\frac{x^2+y^2}{2\epsilon^2} \bigr).
\end{flalign*}
Then we discretize and use the same initial conditions as in the local Gray-Scott model to get figure (\ref{fig2: Solutions of nonlocal Gray-Scott model for $a=0.25$, $b=0.080$, $d_1 = 0.1$, $d_2 = 0.01$.}).

\subsubsection{Visualizations}

\begin{figure}[H]
\centering
     \includegraphics[width=6 cm]{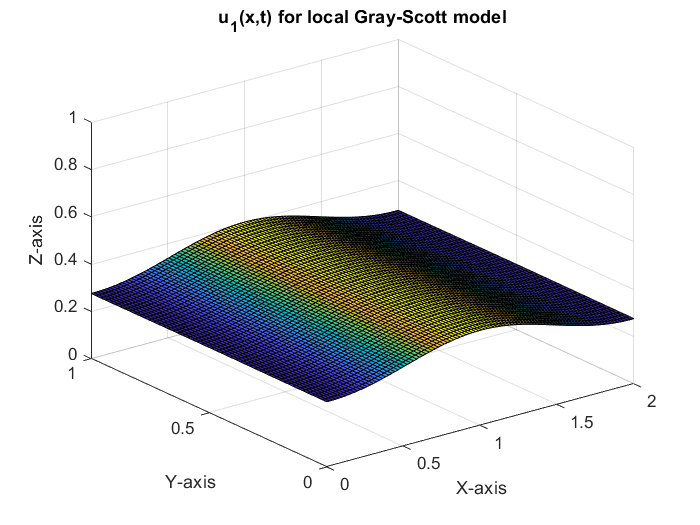}
     \includegraphics[width=6 cm]{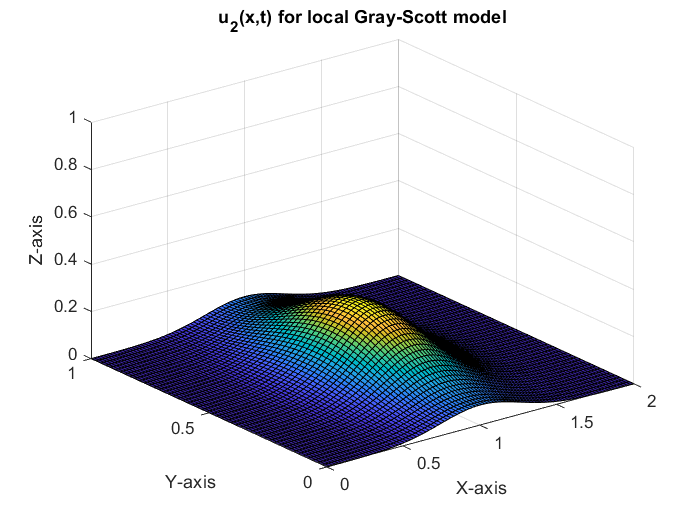}
     \caption{Solutions of local Gray-Scott model for $a=0.25$, $b=0.080$, $d_1 = 0.1$, $d_2 = 0.01$.}
    \label{fig1: Solutions of local Gray-Scott model for $a=0.25$, $b=0.080$, $d_1 = 0.1$, $d_2 = 0.01$.}
\end{figure}
\vspace{-0.1 in}
\begin{figure}[H]
\centering
     \includegraphics[width=6 cm]{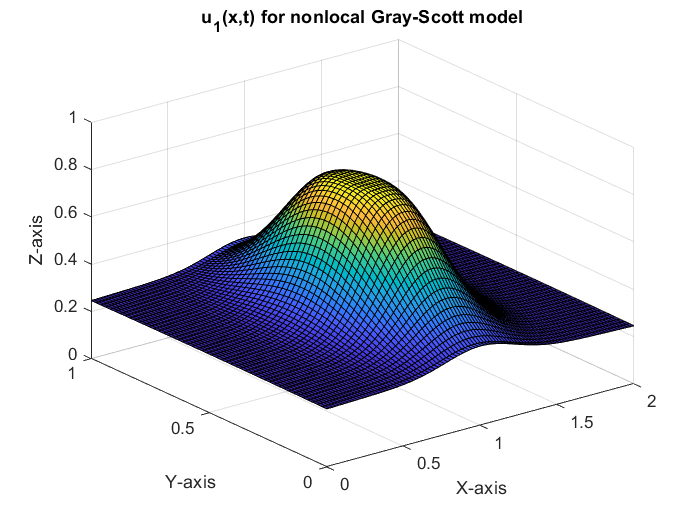}
     \includegraphics[width=6 cm]{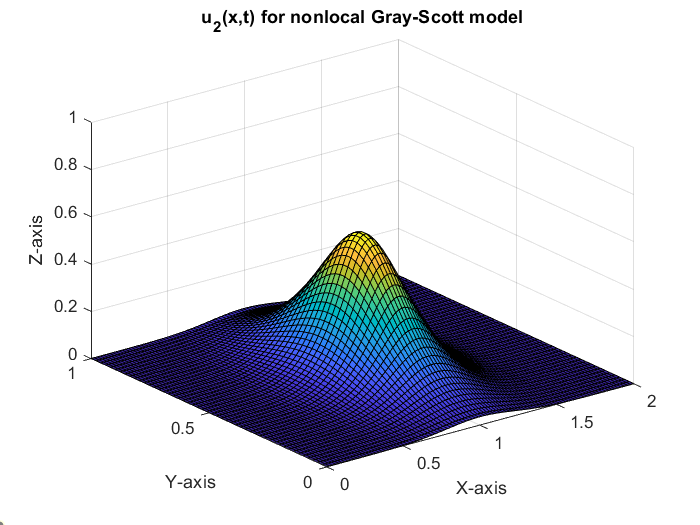}
     \caption{Solutions of nonlocal Gray-Scott model for $a=0.25$, $b=0.080$, $d_1 = 0.1$, $d_2 = 0.01$.}
    \label{fig2: Solutions of nonlocal Gray-Scott model for $a=0.25$, $b=0.080$, $d_1 = 0.1$, $d_2 = 0.01$.}
\end{figure}

\vspace{-0.05 in}
\subsection{Numerical Convergence of Solutions of NRDEs to RDEs}
\vspace{-0.05 in}

In this section, we numerically illustrate Theorem \ref{DiffLimit}. First we need to construct a sequence of solutions of $(\ref{eq:6.2})$ with nonlocal diffusion and reaction terms. We use the same discretization with same initial conditions. Then, we use $(\ref{eq:2.17})$ in such a way that if $j \to \infty$, the solutions of $(\ref{eq:6.2})$ converge to the solutions of $(\ref{eq:6.1})$. We employed this idea into the MATLAB code of our nonlocal Gray-Scott model. This gave the following visual approximations, and from these we have listed the numerical agreement in the tabular form. Since we have a domain of $5000$ discretized points, we have listed the approximated solutions of the first five of them in the table for different values of $j$. We have compared them with the numerical values of solutions of the local system $(\ref{eq:6.1})$.

\vspace{-0.15 in}
\subsubsection{Visualization and Numerical Agreement}
\begin{figure}[H]
\centering
     \includegraphics[width=6 cm]{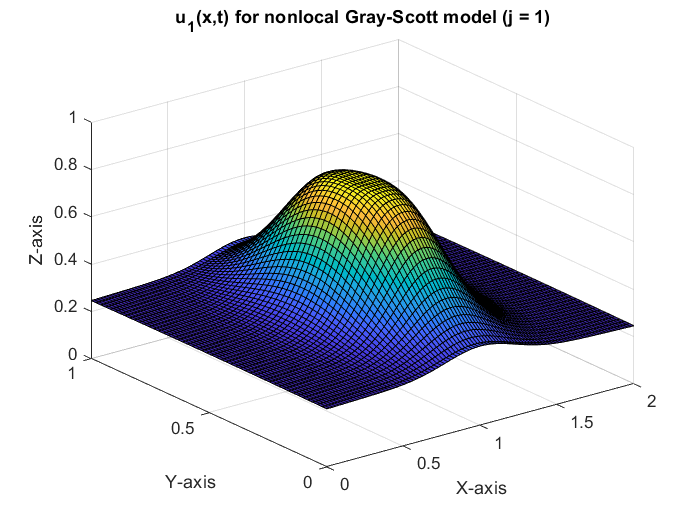}
     \includegraphics[width=6 cm]{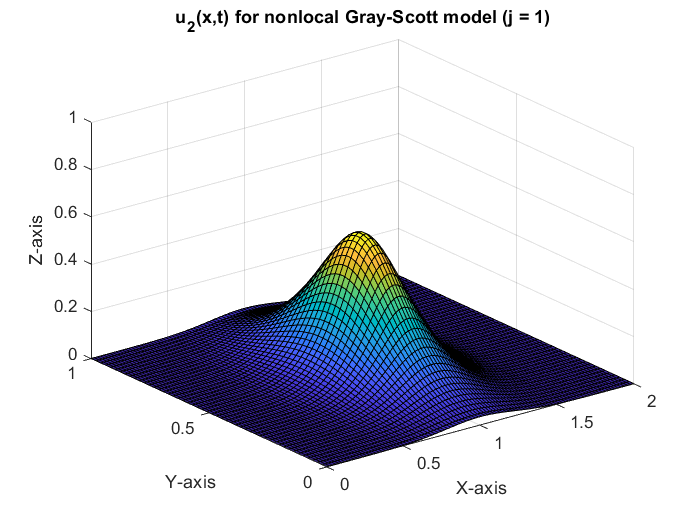}
     \caption{Solutions of nonlocal Gray-Scott model for $j=1$, $a=0.25$, $b=0.080$, $d_1 = 0.1$, $d_2 = 0.01$.}
    \label{fig3: Solutions of nonlocal Gray-Scott model for $j=1$, $a=0.25$, $b=0.080$, $d_1 = 0.1$, $d_2 = 0.01$.}
\end{figure}
\vspace{-0.2 in}
\begin{table}[H]
\centering
\begin{tabular}{l | l | l | l | l | l}
\hline
$u_{1_{local}}$ & $u_{2_{local}}$ & $u_{1_{nonlocal}}$ & $u_{2_{nonlocal}}$ & $ \vert u_{1_{local}}-u_{1_{nonlocal}} \vert $ & $ \vert u_{2_{local}} - u_{2_{nonlocal}} \vert $ \\
\hline
$0.27568$ & $1.5842e-04$ & $0.23263$ & $6.3166e-04$ & $4.305e-02$ & $4.7324e-04$\\
\hline
$0.27569$ & $1.6123e-04$ & $0.23273$ & $6.3795e-04$ & $4.296e-02$ & $4.7672e-04$\\
\hline
$0.2757$ & $1.668e-04$ & $0.23283$ & $6.4412e-04$ & $4.287e-02$ & $4.7732e-04$\\
\hline
$0.27573$ & $1.7503e-04$ & $0.23292$ & $6.5017e-04$ & $4.281e-02$ & $4.7514e-04$\\
\hline
$0.27577$ & $1.8577e-04$ & $0.23301$ & $6.5609e-04$ & $4.276e-02$ & $4.7032e-04$\\
\hline
\end{tabular}
\caption{Top 5 rows of absolute difference of local and nonlocal solutions of Gray-Scott model for $j=1$, $a=0.25$, $b=0.080$, $d_1 = 0.1$ and $d_2 = 0.01$.}
\label{tab1:top}
\end{table}
\vspace{-0.15 in}

\noindent We observe that for $j=1$, $u_2$ is very good compared to $u_1$. We then simulate the program for $j=3$ to get the following results. This time the difference on $u_1$ is better than last time but the difference of $u_2$ is bigger. 
\vspace{-0.1 in}
\begin{figure}[H]
\centering
     \includegraphics[width=6 cm]{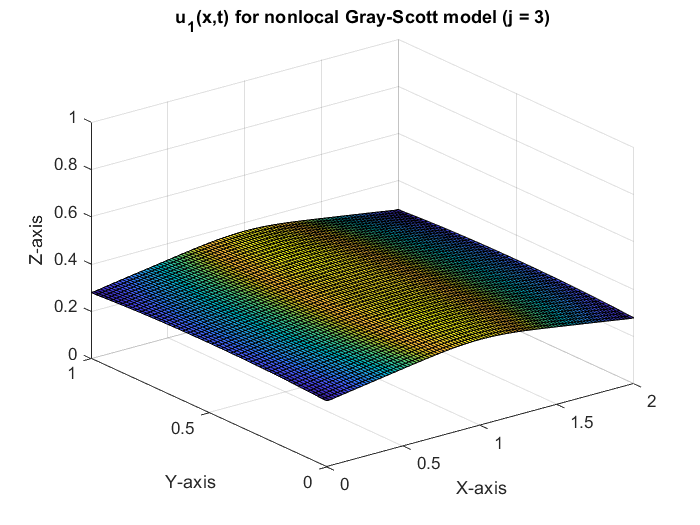}
     \includegraphics[width=6 cm]{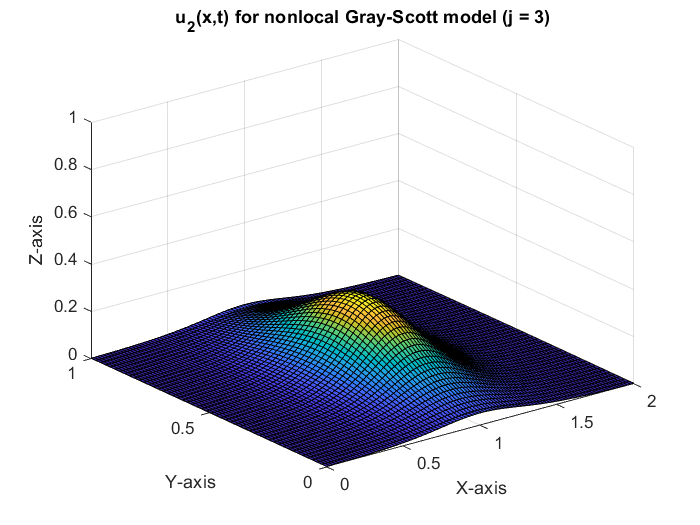}
     \caption{Solutions of nonlocal Gray-Scott model for $j=3$, $a=0.25$, $b=0.080$, $d_1 = 0.1$, $d_2 = 0.01$.}
    \label{fig4: Solutions of nonlocal Gray-Scott model for $j=3$, $a=0.25$, $b=0.080$, $d_1 = 0.1$, $d_2 = 0.01$.}
\end{figure}

\vspace{-0.2 in}
\begin{table}[H]
\centering
\begin{tabular}{l | l | l | l | l | l}
\hline
$u_{1_{local}}$ & $u_{2_{local}}$ & $u_{1_{nonlocal}}$ & $u_{2_{nonlocal}}$ & $ \vert u_{1_{local}}-u_{1_{nonlocal}} \vert $ & $ \vert u_{2_{local}} - u_{2_{nonlocal}} \vert $ \\
\hline
$0.27568$ & $1.5842e-04$ & $0.27984$ & $1.9227e-03$ & $4.16e-03$ & $1.76428e-03$\\
\hline
$0.27569$ & $1.6123e-04$ & $0.2813$ & $2.0285e-03$ & $5.61e-03$ & $1.86727e-03$\\
\hline
$0.2757$ & $1.668e-04$ & $0.28268$ & $2.1348e-03$ & $6.98e-03$ & $1.968e-03$\\
\hline
$0.27573$ & $1.7503e-04$ & $0.28397$ & $2.2411e-03$ & $8.24e-03$ & $2.06607e-03$\\
\hline
$0.27577$ & $1.8577e-04$ & $0.2852$ & $2.347e-03$ & $9.43e-03$ & $2.16123e-03$\\
\hline
\end{tabular}
\caption{Top 5 rows of absolute difference of local and nonlocal solutions of Gray-Scott model for $j=3$, $a=0.25$, $b=0.080$, $d_1 = 0.1$ and $d_2 = 0.01$.}
\label{tab2:top}
\end{table}
\vspace{-0.15 in}

\noindent To speed up the convergence rate, we define $j$ in a different way to get better results. We had,
\vspace{-0.1 in}
\begin{flalign*}
   \varphi_j = j^{n+2} \exp(-j^2 \frac{x^2+y^2}{2 \epsilon^2}) 
\end{flalign*}
%\vspace{-0.05 in}
Since this is a Gaussian Kernel, sending $\epsilon \to 0$ is going to help $j \to \infty$ faster. We employ this idea to redefine the kernel and then run the MATLAB program to get approximations for which the absolute difference between the solutions of local and nonlocal goes down. Here we see good agreement for $j=7$ and $\epsilon=0.81$.
%\vspace{-0.05 in}
\begin{figure}[H]
\centering
     \includegraphics[width=6 cm]{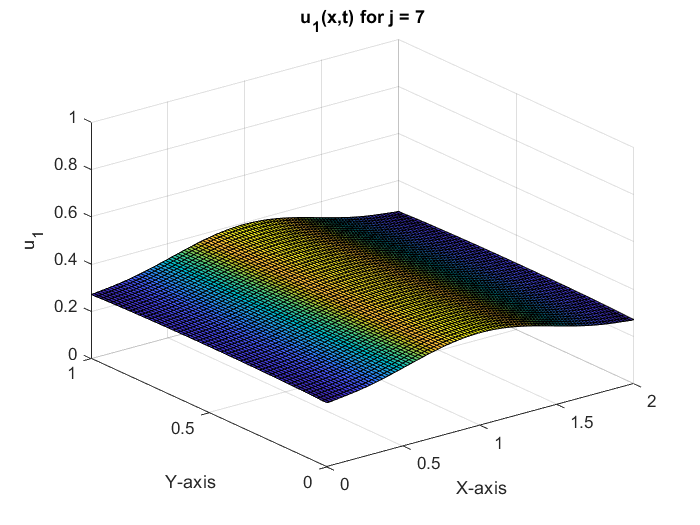}
     \includegraphics[width=6 cm]{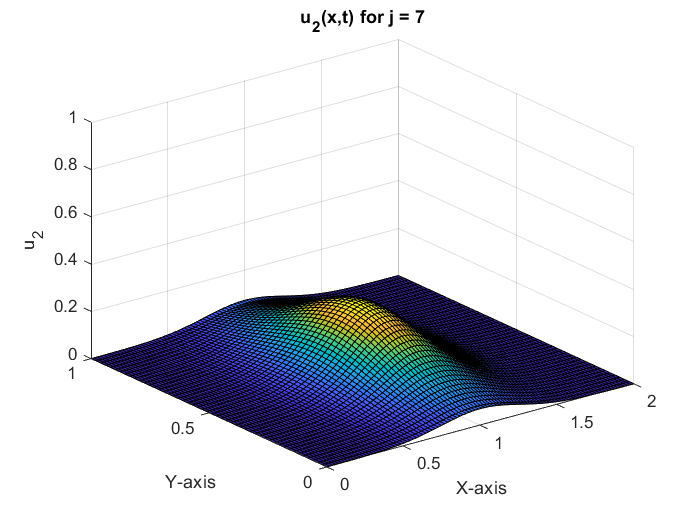}
     \caption{Solutions of nonlocal Gray-Scott model for $j=7$, $\epsilon=0.81$, $a=0.25$, $b=0.080$, $d_1 = 0.1$, $d_2 = 0.01$.}
    \label{fig5: Solutions of nonlocal Gray-Scott model.}
\end{figure}

\begin{table}[H]
\centering
\begin{tabular}{l | l | l | l | l | l}
\hline
$u_{1_{local}}$ & $u_{2_{local}}$ & $u_{1_{nonlocal}}$ & $u_{2_{nonlocal}}$ & $ \vert u_{1_{local}}-u_{1_{nonlocal}} \vert $ & $ \vert u_{2_{local}} - u_{2_{nonlocal}} \vert $ \\
\hline
$0.27568$ & $1.5842e-04$ & $0.2718$ & $2.2286e-04$ & $3.88e-03$ & $6.444e-05$\\
\hline
$0.27569$ & $1.6123e-04$ & $0.27297$ & $2.5215e-04$ & $2.72e-03$ & $9.092e-05$\\
\hline
$0.2757$ & $1.668e-04$ & $0.27394$ & $2.8195e-04$ & $1.76e-03$ & $1.1515e-04$\\
\hline
$0.27573$ & $1.7503e-04$ & $0.27476$ & $3.121e-04$ & $9.7e-04$ & $1.3707e-04$\\
\hline
$0.27577$ & $1.8577e-04$ & $0.27546$ & $3.4252e-04$ & $3.1e-04$ & $1.5675e-04$\\
\hline
\end{tabular}
\caption{Top 5 rows of absolute difference of local and nonlocal solutions of Gray-Scott model for $j=7$, $\epsilon=0.81$, $a=0.25$, $b=0.080$, $d_1 = 0.1$ and $d_2 = 0.01$.}
\label{tab3:top}
\end{table}

\noindent As we can see, the example above numerically illustrates the result in Theorem \ref{DiffLimit}.

\subsection{Numerical Approximation of Solutions of NRDEs with (INT-SUM2).}

Now we consider $\Omega=(0,2) \subset \mathbb{R}$, and we implement the Method of Lines \cite{verwer1984convergence,chen2005nonlinear,mohamed2018chemical} procedure to approximate the solutions of the following system:
\begin{flalign} \label{eq:6.3}
\begin{cases}
\frac{\partial}{\partial t} u(x,t) & = d_1 \Gamma_1 u(x,t) + u(x,t)+v(x,t)-u(x,t)v(x,t)^2, \hspace{1.1 in} (x,t) \in Q_{0,\infty}\\
\frac{\partial}{\partial t} v(x,t) & = d_2 \Gamma_2 v(x,t) + u(x,t)v(x,t)^2, \hspace{2.10 in} (x,t) \in Q_{0,\infty} \\
u(x,0) & =0.5 e^{\frac{-(x-\frac{L}{2})}{0.1}}, \ v(x,0) = 0.5 e^{\frac{-(x-\frac{L}{2})}{0.1}}, \hspace{1.95 in} x \in \Omega.
\end{cases}
\end{flalign}
Since the Method of Lines is a semi-discrete method, we discretize only the spatial variable $x$ and leave the time variable $t$ as it is. We discretize the domain with 100 equal grid points. We choose $d_1=0.1$, $d_2=0.01$, \textit{ode15s} as the ODE solver for stiff differential equations, which uses the Backward Differentiation Formula (BDF), also known as Gear’s method for time integration. We use $e^{-\frac{(x_i - x_j)^2}{2\epsilon^2}}$ as the Gaussian kernel where $x_i$ is the current grid point and $x_j$ are the other grid points for $i\ne j$. We use $\epsilon=1.0$, and we follow the similar procedure as in the previous example to discretize the system in $(\ref{eq:6.3})$, but this time only spatially. Therefore, the system results in the two dimensional ODE system given by
\begin{flalign*}
 \frac{du_{new}}{dt} & = d_1 \sum_{\substack{i=1 \\ i\neq j}}^j e^{-\frac{(x_i-x_j)^2}{2\epsilon^2}} (u(x_i,t)-u(x_j,t)) (x_j-x_{j-1}) + u_{old}+v_{old}-u_{old}v_{old}^2, \\
\frac{dv_{new}}{dt} & = d_2 \sum_{\substack{i=0 \\ i\neq j}}^j e^{-\frac{(x_i-x_j)^2}{2\epsilon^2}} (v(x_i,t) - v(x_j,t)) (x_j-x_{j-1}) + u_{old}v_{old}^2, 
\end{flalign*}
We also discretize our boundary in similar way and run the code for the above model for two seconds to get the following results.

\begin{figure}[H]
\centering
     \includegraphics[width=6 cm]{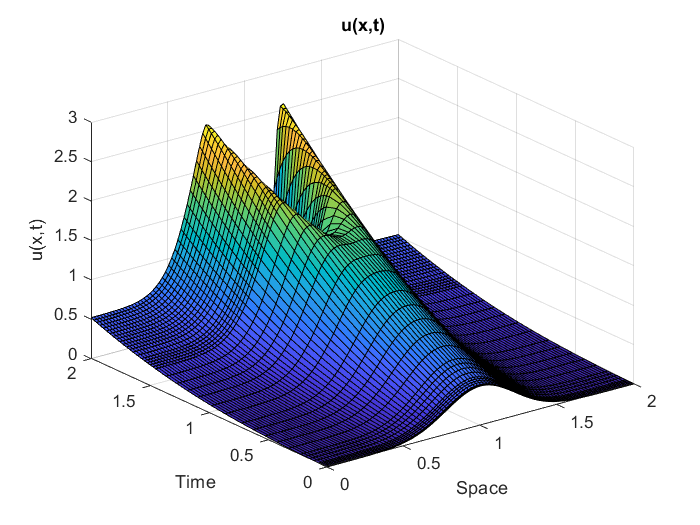}
     \includegraphics[width=6 cm]{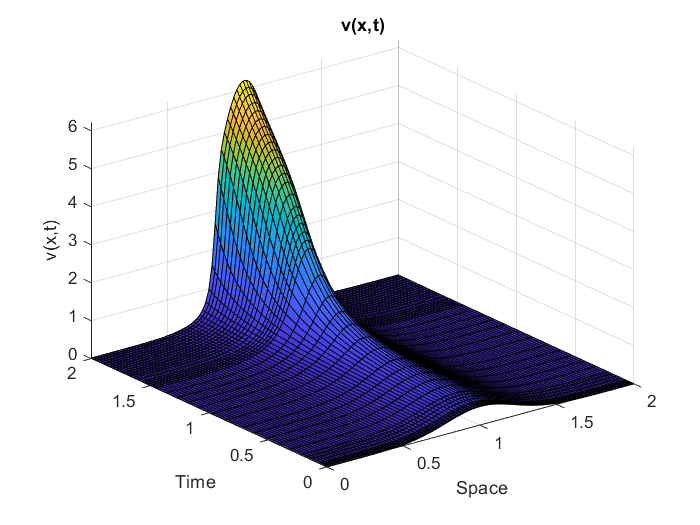}
     \caption{Solutions of NRD model with new intermediate sum inequality.}
    \label{fig6: Solutions of NRD model with new intermediate sum inequality.}
\end{figure}

\noindent Since, nonlocal operators do not regularize solutions, we observe a nonsmooth spike-like shape.

\subsection{Numerical Approximation of Solutions of Mixed type RDEs with (INT-SUM2).}
\vspace{-0.1 in}

We apply the Method of Lines for the following mixed type RDE model with the same domain settings (same boundary condition too) as in the previous section.  
\begin{flalign} \label{eq:6.4}
\begin{cases}
\frac{\partial}{\partial t} u(x,t) & = d_1 \Gamma u(x,t) + u(x,t)+v(x,t)-u(x,t)v(x,t)^2, \hspace{1 in} (x,t) \in Q_{0,\infty} \\
    \frac{\partial}{\partial t} v(x,t) & = d_2 \Delta v(x,t) + u(x,t)v(x,t)^2, \hspace{1.97 in} (x,t) \in Q_{0,\infty} \\
   \frac{\partial v}{\partial \eta}(x,t) & = 0, \hspace{3.33 in} (x,t) \in \partial \Omega \times (0,\infty)\\
    u(x,0) & =0.5 e^{\frac{-(x-\frac{L}{2})}{0.1}}, \ v(x,0) = 0.5 e^{\frac{-(x-\frac{L}{2})}{0.1}}, \hspace{1.78 in} x \in \Omega.
\end{cases}    
\end{flalign}
We discretize the Laplacian using Central Difference Formula on Taylor polynomial but in $1$-D this time. We have got the following results which clearly shows significant qualitative and quantitative differences of solutions from the previous model.
\vspace{-0.1 in}
\begin{figure}[H]
\centering
     \includegraphics[width=5.5 cm]{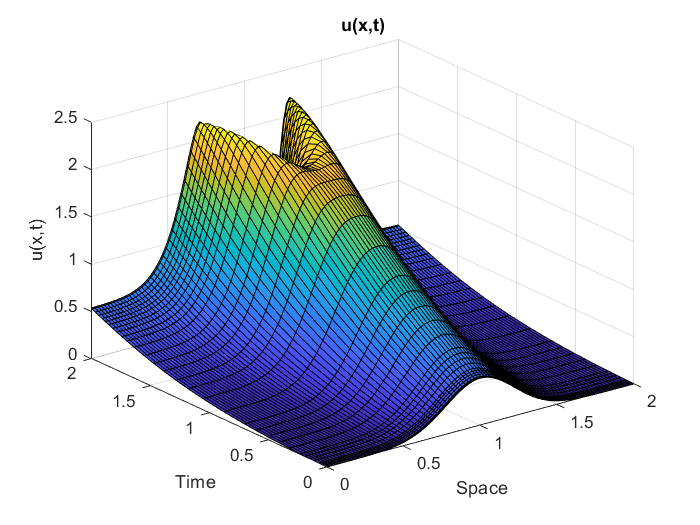}
     \includegraphics[width=5.5 cm]{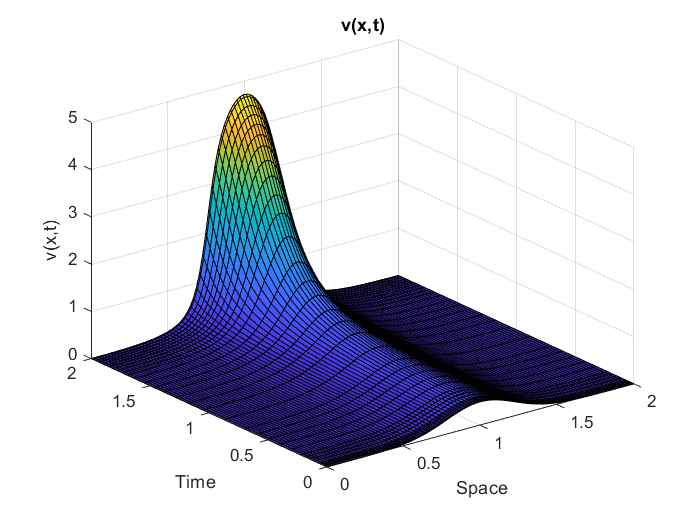}
     \caption{Solutions of Mixed RDE model new intermediate sum inequality.}
    \label{fig7: Solutions of Mixed RDE model new intermediate sum inequality.}
\end{figure}

\noindent We observe that the peaks of the solution $v$ of this system are less pronounced and smoother than those for $v$ in the system represented by $(\ref{eq:6.3})$. This is because the system $(\ref{eq:6.4})$ has Laplacian operator (local diffusion) in the PDE for $v$ which smooths solutions. As a result, figure $(\ref{fig7: Solutions of Mixed RDE model new intermediate sum inequality.})$ illustrates more smoothness in $v$ than figure $(\ref{fig6: Solutions of NRD model with new intermediate sum inequality.})$.

\section{Conclusions}

We establish global existence of solutions for a general class of $m$-component nonlocal reaction–diffusion systems posed on bounded domains in dimension $n \ge 2$ that preserve nonnegativity and control mass. We also obtain diffusive limit results by employing a linear intermediate sum condition to obtain bounds for $\|u(\cdot,t)\|_{p,\Omega}$ with $2\le p<\infty$ and $0 \leq t \leq T$ depending on $p$, $T$, initial data and parameters in our assumptions, but independent of the kernel $\varphi$ of the nonlocal operators in \eqref{eq:2.11}. Examples and numerical approximations are included to illustrate our results.\\
\\

\noindent\textbf{Conflict of Interest:} The authors declare no conflict of interest.\\

\noindent\textbf{Use of AI Tools Declaration:} The authors did not make use of AI to generate any of the content in this work.

\bibliography{thesis_biblio}
\bibliographystyle{acm}

\end{document}